\documentclass{article}	
\usepackage{graphicx}
\usepackage[utf8]{inputenc}
\usepackage{indentfirst,csquotes}

\usepackage[bottom=1.5in]{geometry}
\usepackage[english]{babel}
\usepackage{tikz}
\usepackage{fancyhdr} 
\usepackage{amsthm}
\usepackage{nicematrix}

\usepackage{hyperref}
\usepackage{amssymb}
\usepackage{mathtools}
\usepackage{amsmath}
\usepackage{latexsym}
\usepackage{esint}
\usepackage{wasysym}
\usepackage{pst-eucl}
\usepackage{comment}
\usepackage{hyperref}
\usepackage{blkarray}

\hypersetup{
    colorlinks=teal,
    linkcolor=blue,
    citecolor=teal
    }

\newcommand{\ps}[2]{\left<#1,#2\right>}

\newcommand{\norm}[1]{\left\lVert#1\right\rVert}
\newcommand{\nor}[1]{\left\lvert#1\right\rvert}

\newtheorem{Theorem}{Theorem}
\numberwithin{Theorem}{section}
\newtheorem{Proposition}[Theorem]{Proposition}
\newtheorem{Lemma}[Theorem]{Lemma}
\newtheorem{Corollary}[Theorem]{Corollary}
\theoremstyle{definition}
\newtheorem{Definition}[Theorem]{Definition}
\theoremstyle{remark}
\newtheorem{Remark}[Theorem]{Remark}
\numberwithin{equation}{section}

 \DeclareMathOperator{\R}{\mathbb{R}}

\DeclareMathOperator*{\spec}{spec}
\DeclareMathOperator*{\tr}{tr}
\newcommand*{\Mp}{\mathcal{M}_{K}^{+}}
\newcommand{\Mm}{\mathcal{M}_{K}^{-}}
\newcommand{\p}{\partial}
\newcommand{\h}{\mathcal{H}}

\newcommand{\osc}{\operatorname{osc}}

\usepackage{dirtytalk}

\begin{document}
\title{Free boundary regularity for the inhomogeneous one-phase Stefan problem} 

\author{Fausto Ferrari
  \and
  Nicol\`o Forcillo
  \and
  Davide Giovagnoli
  \and 
  David Jesus
} 

\author{Fausto Ferrari
\and 
Nicol\`o Forcillo
\and
Davide Giovagnoli
\and
David Jesus
}
\newcommand{\Addresses}{{
  \bigskip
  \footnotesize

  Fausto Ferrari, 
   \textit{E-mail address:} {\tt fausto.ferrari@unibo.it}

  \vspace{0.1cm}
 Davide Giovagnoli,
    \textit{E-mail address:} {\tt d.giovagnoli@unibo.it}
  
   \vspace{0.1cm}
  \noindent\textsc{Dipartimento di Matematica,
Universit\`a di Bologna\\ 
\noindent Piazza di Porta San Donato 5, 40126 Bologna, Italy}\\

  \medskip
  Nicol\'o Forcillo,  \textit{E-mail address:} {\tt forcill1@msu.edu} 
  \vspace{0.1cm}
  \par\nopagebreak \noindent \textsc{Departments of Mathematics, Michigan State University, 619 Red Cedar Road  \\ \noindent East Lansing, MI 48824, USA}\\

  \medskip


 David Jesus,  \textit{E-mail address:} {\tt david.dejesus@kaust.edu.sa} 
 
 \vspace{0.1cm}
 \noindent \textsc{Applied Mathematics and Computational Sciences (AMCS), \\ King
Abdullah University  of Science and Technology (KAUST), \\  Thuwal 23955-6900, Kingdom of Saudi Arabia}
  
   }}

\date{\today}

\maketitle
{\bf Abstract.} In this paper, we prove that flat free boundaries of  solutions of the inhomogeneous one-phase Stefan problem are $C^{1,\alpha}$. The method consists of employing a hodograph transform and deriving the regularity via a linearization technique, following the approach introduced by De Silva, Forcillo, and Savin in \cite{DFS23}.
\thispagestyle{fancy}
\fancyhead{} 
\fancyfoot{}
\fancyfoot[L]{\footnoterule {\small 
Keywords: Inhomogeneous Stefan problem, free boundary problems, flat free boundaries, parabolic operators, viscosity solutions. \\ MSC: 35R35, 35K55, 80A22}
}
 
\section{Introduction}
The goal of this paper is to introduce new results concerning the regularity of flat free boundaries  in a one-phase Stefan problem governed by an inhomogeneous equation. More precisely, we deal with nonnegative viscosity solutions $u: \Omega \times (0,T] \mapsto \mathbb{R}$ of the free boundary problem,
    \begin{equation}\label{eq:StefanNH}
        \begin{cases}
          \partial_{t} u - \Delta u  =  f& \text{in} \,\, \Omega^+(u):=\left(\Omega \times (0,T]\right) \cap \{ u > 0  \},\\
       \partial_{t} u = \nor{\nabla u}^2& \text{on} \,\, \mathcal{F}(u):=\left(\Omega \times (0,T] \right) \cap \partial\{ u > 0  \},
    \end{cases}
    \end{equation}
   where $\Omega$ is a bounded domain in $\mathbb{R}^{n}$, $T > 0$,  $f: \Omega \times [0,T] \mapsto \mathbb{R}, \, f \in C(\Omega \times [0,T]) \cap L^{\infty }(\Omega \times [0,T])$. \(\mathcal{F}(u)\) is the so-called \textit{free boundary} of the solution.

    The Stefan problem consistently attracts attention since it models fundamental physical phenomena, especially the processes of ice melting or liquid water freezing. This topic plays an important role within the context of free boundary problems.
In literature, the Stefan problem has been treated in several ways, so, we refer the interested reader to \cite{vis,Friedman,Friedmanerr,CS} for a detailed discussion on the various formulations and results. 

We are interested in the free boundary regularity of viscosity solutions of the problem, as in \cite{ACS,ACS2,ACS3,DFS23}. This topic had been treated in the trilogy \cite{ACS,ACS2, ACS3}, which collects important contributions to the free boundary regularity in the homogeneous Stefan problem. In these papers, the authors obtained sharp results in the two-phase setting. 
Further achievements regarding the existence of analytic solutions can be found in \cite{PrSaSi}. We refer to \cite{ChKi} for the relationship between the solutions' smoothness and the initial data regularity. In addition, \cite{KiPo} focused on the existence of a solution, while, in \cite{FeSa_p}, operators with variable coefficients have been studied, exploiting the same techniques of \cite{ACS,ACS2, ACS3}. 
 
Recently, in \cite{DFS23}, a new approach for investigating the free boundary regularity has been introduced. Specifically, the authors addressed the regularity for the homogeneous one-phase Stefan problem by employing the {\it hodograph transform} jointly with the viscosity theory, constructing a notion of  flat free boundary within an evolving framework. This approach follows the idea introduced in \cite{DeSilva2011} in the elliptic scenario, and then widely applied in different cases, see, for instance, \cite{DFS1,DeFS2,DeFS3} and  \cite{FL,FL2,FL3,FL4} for recent results in the elliptic degenerate nonlinear setting, see also \cite{FLS,FJL} for comprehensive surveys. 

In the parabolic setting, fully nonlinear operators in the homogeneous framework have been recently investigated in \cite{wang2024free}. We also mention \cite{KS}, where a parabolic homogeneous two-phase problem with a transmission-type free boundary condition has been considered. In that work, the free boundary condition involves only spatial derivatives, in contrast to the one appearing in the Stefan problem.
    
    In our paper, we deal with the inhomogeneous one-phase case \eqref{eq:StefanNH}, which models the melting of the ice in the presence of an external source or a sink of heat, following the approach introduced in \cite{DFS23}. 
More precisely, we are interested in the regularity of the free boundary starting from a geometric condition, which we call $\varepsilon$-flatness and introduce in the next definition.

\begin{Definition}\label{def:flat_sol}
 For $K>1$, $0<\lambda \leq 1$, and $\varepsilon>0$, we say  that $u$ is $\varepsilon$-flat in $B_{2\lambda} \times (-2K^{-1}\lambda,0]$ if there exist two smooth functions $a_n,b:\R \to \R$ such that
 \begin{equation} \label{Flatness}
   a_n(t) \, \left (x_n - b(t)-\varepsilon \lambda \right)^+  \le u(x,t)  \le a_n(t) \, \left(x_n - b(t)+\varepsilon \lambda \right)^+,\quad (\varepsilon\text{-flatness})
 \end{equation}
  with
 \begin{equation} \label{hp:flattt}
 	c(1+\|f^- \|_\infty) \le  a_n \le C, \quad \quad |a_n'(t)| \le \lambda^{-2}, \quad \quad b'(t)=- a_n(t),\quad\quad  \|f\|_{L^\infty}\leq \varepsilon^2,
 \end{equation}
 where $c,\, C>0$ depend only on $K$ and $n$. Here, $a_n'$ and $b'$ denote the derivatives of $a_n$ and $b$, respectively.
 \end{Definition}
 
 We remark that the assumption $b'(t)=-a_n(t)$ in Definition \ref{def:flat_sol} implies that the approximating function $a_n(t)(x_n-b(t))^+$ satisfies the free boundary condition. Moreover, we emphasize that Definition \ref{def:flat_sol} differs from the more common notion of $\varepsilon_0$-flatness of the free boundary, which we recall below.
    \begin{Definition}\label{Def:flat_FB}
    Let $\partial_x\{u(\cdot,t)>0\}$ denote the spatial boundary in $\R^n$ of $\{u(\cdot,t)>0\}$. We say that $\partial_x\{u>0\}$ is $\varepsilon_0$-flat in $ B_\lambda$ if, for every time $t$, there exists a direction $\nu$ such that
    \[
    \partial_x\{u(\cdot, t)>0\}\cap  B_\lambda\subset \{|\ps{ x}{\nu}|\leq \varepsilon_0\lambda\}
    \]
    and
    \begin{align*}
        u=0 \quad &\textit{ in } \quad \{\ps{ x}{\nu} < -\varepsilon_0\lambda\},\\
        u>0 \quad &\textit{ in } \quad \{\ps{ x}{\nu} > \varepsilon_0\lambda\}.
    \end{align*}
\end{Definition}
We observe that a key difference from the elliptic framework is the definition of flatness of the free boundary, due to its evolution in time.

When $f \equiv 0$, Definition \ref{def:flat_sol} reduces to the assumption of Theorem 1.2 in \cite{DFS23}. For the sake of clarity, in the sequel,  we always refer to Definition \ref{def:flat_sol} when we mention the $\varepsilon$-flatness of $u$.

  In \cite[Lemma 2.1]{DFS23}, the authors show that the geometric notion of flatness  for the free boundary, stated in Definition \ref{Def:flat_FB},  combined with a nondegeneracy condition, implies the $\varepsilon$-flatness of the solution near the free boundary.
In our inhomogeneous case,  the proof of this fact is provided in \cite[Theorem 1.1]{FGJ}.

     The main result of this paper is the following.

\begin{Theorem}\label{Theorem:flatsolution}
    Let $K>1$. There exists  $0 <\lambda \leq 1$ such that if $u$ is a viscosity solution of the one-phase Stefan problem \eqref{eq:StefanNH} in $B_{2\lambda} \times (-2K^{-1}\lambda,0]$, $0\in \mathcal{F}(u),$   and $u$ is $\varepsilon$-flat in the sense of Definition \ref{def:flat_sol} for some small $\varepsilon$ only depending on $K$ and $n$, then the free boundary $\mathcal{F}(u)$ is a $C^{1,\alpha}$-graph in the $x_n$-direction in $B_\lambda \times[-K^{-1}\lambda ,0]$. 
\end{Theorem}
 In other words, we prove the $C^{1,\alpha}$-regularity of flat free boundaries for solutions to \eqref{eq:StefanNH}.

 We  point out that it is enough to prove Theorem \ref{Theorem:flatsolution} under the  assumptions
\begin{equation} \label{eq:relaxed_hp}
\lambda \leq \lambda_0 \quad \text{and} \quad |a_n'(t)| \le c_0\lambda^{-2},
\end{equation}
where $\lambda_0$ and $c_0$ are sufficiently small constants depending on $K$ and $n$. This reduction can be achieved by considering the problem in a domain of size $\tau \lambda$ for $\tau$ small enough, and by relabeling $\tau \lambda$ and $\varepsilon \tau^{-1}$ as $\lambda$ and $\varepsilon$, respectively. 


The proof of Theorem \ref{Theorem:flatsolution} follows the approach in \cite{DFS23}.   After applying the hodograph transform to solutions of \eqref{eq:StefanNH}, the problem formally reduces to studying viscosity solutions of
\begin{equation}\label{eq:QuasiStefanNH}
        \begin{cases}
        \partial_{t} u - \tr( A(\nabla u) D^2 u) + (\partial_{x_n}u) \,  f(x',u,t)=0  &\text{in} \,\, \{ x_n > 0 \},\\
       \partial_{t} u = g(\nabla  u) &\text{on} \,\,  \{ x_n = 0 \},
    \end{cases}
\end{equation} 
where $A$ and $g$ are defined in \eqref{eq:HodographStefan} and \eqref{eq:g}.

 The hodograph transform plays an important role in our analysis. Since it transforms the free boundary into a flat plane, it removes the main  challenge in the study of free boundary problems. However, this difficulty is transferred into the solution itself since, in a small neighborhood of this flat interface, it might be multi-valued, thus failing to be a well-defined function. To address this subtle and important issue, we need to deal with the very weak notion of multi-valued viscosity  solution, made precise in Definition \ref{def:multivalue_sol}. Therefore, the problem of studying the free boundary regularity of \eqref{eq:StefanNH} turns out to follow from the regularity up to the \textit{fixed} boundary of solutions to \eqref{eq:QuasiStefanNH}. The proof of this last property relies on an improvement of flatness result, see Proposition \ref{Prop:ImprovFlat}.

The hodograph transform introduces some extra difficulties in a small neighborhood of the free boundary, as already faced in \cite{DFS23}. The problem \eqref{eq:QuasiStefanNH} is different from the one in \cite{DFS23} due to the presence of  the gradient term $(\partial_{x_n}u) \,  f(x',u,t)$.
This fact introduces a further challenge in recovering the Harnack inequality, as in the homogeneous case it follows from the usual parabolic Harnack inequality. Specifically, we apply a more sophisticated Harnack inequality for the following extremal equations
\begin{align*}
    \partial_t u-\mathcal{M}^\pm(D^2u)\mp\Lambda|\nabla u|+f=0,
\end{align*}
which can be found in \cite{koike2019weak}.

Even though the equation in \eqref{eq:QuasiStefanNH} does not fit the homogeneous framework, we prove that the achieved limiting profile is the same as the one obtained in the homogeneous case.
Consequently, we can exploit the Schauder-type estimates, obtained in \cite{DFS23}, with respect to an appropriate distance, which encodes features of both the parabolic and hyperbolic rescaling of the problem. The improved regularity for the linearized problem is then transferred to \eqref{eq:QuasiStefanNH} by a comparison principle. 

Let us mention that the argument developed in this paper can be combined with the one introduced in \cite{wang2024free} to consider a fully nonlinear PDE in \eqref{eq:StefanNH} governing the Stefan problem.

	Now, we describe the structure of the paper. In Section \ref{Section:preliminaries}, we provide the necessary notation, definitions, and further remarks on the lack of scaling of the problem. In Section \ref{Section:hodograph}, we apply the hodograph transform to change equation \eqref{eq:StefanNH} into \eqref{eq:QuasiStefanNH} and discuss how the initial assumptions translate into this new context. We end this section by stating the main improvement of flatness result. We devote Section \ref{Section:propertiesv} to study properties of the error function, which is defined as $u-\ell,$ where $\ell$ is the affine approximation of $u$, for each time $t$. In Section \ref{Section:C1a}, the proofs of the improvement of flatness and Theorem \ref{Theorem:flatsolution} are presented.



    \section{Preliminaries}\label{Section:preliminaries}
    \subsection{Notation and definitions} \label{Subsection:notation}
In this subsection, we introduce our main notation. As usual, given $(x,t)\in \R^{n+1}$, we call $x\in \mathbb{R}^n$ the space variable and $t \in \R$  the time variable. A point $x\in \mathbb{R}^n$ will sometimes be written as $x=(x',x_n)$, where $x'\in \R^{n-1}$ and $x_n \in \R$.
For a function $u(x,t)$, we denote by $\nabla u$ the gradient of $u$ with respect to $x$, $D^2 u$ the Hessian with respect to $x$, and $\nabla_{x'} u$ the gradient with respect to $x'$. 
A \textit{parabolic cylinder} $\mathcal{P}$ is any set of the form
 $$\mathcal{P}=U \times (t_1,t_2],$$ 
 where $U$ is a smooth bounded open domain in $\R^n$ and $t_1 < t_2$. We denote by $\partial_p \mathcal{P}$ the \textit{parabolic boundary} of $\mathcal{P}$, i.e., $\partial_p {\mathcal{P}}=\left( U\times \{t_1\}  \right)  \cup   \left( \partial U \times [t_1,t_2] \right) $.
 
Given $(x,t), (y,s)\in \R^{n+1}$, we define the \textit{parabolic distance} as
	\begin{equation}\label{eq:parab-dist}
	d_p((x,t), (y,s)) :=\big(|x-y|^2 + |t-s|\big)^{1/2}.
	\end{equation}
Let $C^{0}(\mathcal{P})$ be the set of continuous functions in ${\mathcal{P}}$. For $\alpha \in (0,1]$ and $u\in C^{0}(\mathcal{P})$, we define the parabolic H\"older semi-norm as
\begin{equation*}
    [u]_{C^{\alpha}(\mathcal{P})} := \sup_{\substack{(x,t)\neq (y,s)}} \frac{\nor{u(x,t)-u(y,s)}}{d_p((x,t),(y,s))^{\alpha}},
\end{equation*}
and the parabolic H\"older norm as
\begin{equation*}
    \norm{u}_{C^{\alpha}(\mathcal{P})} := \norm{u}_{L^{\infty}(\mathcal{P})} +[u]_{C^{\alpha}({\mathcal{P}})}.
\end{equation*}
 Furthermore, we define the H\"older semi-norm in time to be
\begin{equation*}
[u]_{C_t^{\alpha}({\mathcal{P}})} := \sup_{ \substack{(x,t_1)\neq (x,t_2)}} \frac{|u(x,t_1)-u(x,t_2)|}{|t_1-t_2|^\alpha}.
\end{equation*}
This is necessary to introduce the $C^{1,\alpha}$-norm in ${\mathcal{P}}$
\begin{equation*}
    \norm{u}_{C^{1,\alpha}({\mathcal{P}})}:=\norm{u}_{L^{\infty}({\mathcal{P}})} + \norm{\nabla u}_{L^{\infty}({\mathcal{P}})} +[\nabla u]_{C^{\alpha}({\mathcal{P}})} + [u]_{C_t^{(1+\alpha)/2}({\mathcal{P}})}.
\end{equation*}
We denote by $C^{\alpha}({\mathcal{P}})$ (resp. $C^{1,\alpha}({\mathcal{P}})$) the space of continuous (resp. differentiable in space) functions in $\mathcal{P}$, which are bounded in the $\norm{\cdot}_{C^{\alpha}({\mathcal{P}})} $ (resp. $\norm{\cdot}_{C^{1,\alpha}({\mathcal{P}})}$) norm.

 In the following, for simplicity, we consider parabolic cylinders of the form
\begin{equation}\label{eq:parab-cylin}
	\mathcal{P}_r(x_0,t_0):=Q_{r}(x_0) \times (t_0-r^2,t_0],
\end{equation}
where $n \geq 2$, given $r > 0$, we set
$$Q_{r} := (-r,r)^{n}, \quad Q_{r}^{+}:=  Q_{r} \cap \{ x_n \geq 0 \},  \quad  Q_{r}(x_0):=  x_0 + Q_{r}.$$
In addition, we denote by
\begin{equation}\label{eq:C_r-F_r}
    \begin{aligned}
        &\mathcal{C}_r := (Q_r \cap \{x_n >0 \}) \times (-r,0], \\
        &\mathcal{F}_r:=\{(x,t) | \,  x \in Q_r \cap \{x_n = 0 \}, \, t \in (-r,0] \}.
    \end{aligned}
\end{equation}
It will be useful to define the \textit{Dirichlet boundary} of $\mathcal{C}_1 $ as
\begin{equation*}
    \p_D \mathcal{C}_1:= \partial \mathcal{C}_1 \cap \left( \{ t=-1\} \cup \{ x_n=1\} \cup_{i=1}^{n-1} \{ \nor{x_i}=1\} \right),
\end{equation*}
which, together with $\mathcal{F}_1,$ represents the parabolic boundary of $\mathcal{C}_1$. \\
Lastly, for $K>1$, we recall the definition of extremal Pucci operators with ellipticity constants $K^{-1}$ and $K$  
\begin{equation*}
	\mathcal M_K^+ (N) = \max_{K^{-1} I \le A \le K I} \quad \tr( A N), \quad \quad  \quad \mathcal M_K^- (N) = \min_{K^{-1} I \le A \le K I} \quad \tr( A N).
\end{equation*}

We now provide  the definition of viscosity solution of \eqref{eq:StefanNH}.
    \begin{Definition}\label{Def:viscosity_stefan}
     We say that  a continuous function $u$ is a viscosity solution of \eqref{eq:StefanNH} in $\Omega \times (0,T]$ if it is nonnegative and its graph cannot be touched by above (resp. below) at a point $(x_0,t_0) \in ,$ in a parabolic cylinder $ \mathcal{P}_r(x_0,t_0) \subset \Omega \times (0,T],$
      by the graph of a classical strict supersolution $\varphi^+$ (resp. subsolution). 
      By a classical strict supersolution in $\mathcal{P}$, we mean that $\varphi(x,t) \in C^{2}( \mathcal{P}_r(x_0,t_0))$, $\nabla \varphi \neq 0,$ and it solves
    \begin{equation*}
        \begin{cases}
       \partial_{t} \varphi -\Delta \varphi > f& \text{in} \,\,  \mathcal{P}_r(x_0,t_0)\cap \{ \varphi > 0  \},\\
       \partial_{t} \varphi > \nor{\nabla \varphi}^2& \text{on} \,\,  \mathcal{P}_r(x_0,t_0) \cap \partial\{ \varphi > 0  \}.
    \end{cases}
    \end{equation*}
    Similarly, we can define a strict classical subsolution.
    \end{Definition}

We want to relate our inhomogeneous equation to the homogeneous one studied in \cite{DFS23}. Hence, we start by performing a rescaling to ensure that we can assume, without loss of generality, that the source term $f$ is arbitrarily small. Specifically, if $u$ is a viscosity solution of \eqref{eq:StefanNH} in  $\mathcal{P}_1(0,0)$, see \eqref{eq:parab-cylin}, calling
\[
    \Lambda := \frac{\|f\|^\frac{1}{2}_{L^\infty}}{\varepsilon_1},
\]
then the function
\[
\bar u(x,t)=u(\Lambda^{-1}x, \Lambda^{-2} t), \quad (x,t)\in  \mathcal{P}_\lambda(0,0),
\]
is a solution of 
\begin{align*}
    \begin{cases}
          \partial_{t} \bar u - \Delta \bar u  =  \bar f& \text{in} \,\,  \mathcal{P}_\lambda(0,0) \cap \{ \bar u > 0  \},\\
       \partial_{t} \bar u = \nor{\nabla \bar u}^2& \text{on} \,\,  \mathcal{P}_\lambda(0,0) \cap \partial\{ \bar u > 0  \},
    \end{cases}
\end{align*}
with $\|\bar f\|_{L^\infty}\leq \varepsilon_1^2$. For ease of notation, we will refer to this function as $u$.

\subsection{Some remarks about the scalings of the problem}\label{differentrescaling}
We recall that, when dealing with a problem such as \eqref{eq:StefanNH}, a dichotomy often arises in the application of the rescaling argument, due to the differing scaling behaviors of the PDE and the free boundary condition, as already noted in \cite{ACS,ACS2, ACS3}. 
Indeed, the Stefan problem exhibits a natural mixed scaling: parabolic in the interior and hyperbolic near the free boundary. 

If we apply the hyperbolic rescaling, for instance, centered at $(0,0)\in \mathcal{F}(u)$, given by
\begin{equation*}
    (x,t) \mapsto u_{\lambda}^{H}(x,t)\coloneqq \frac{u(\lambda x, \lambda t )}{\lambda}, \quad (x,t) \in B_1 \times (-1,0],
\end{equation*}
we obtain that  $u_{\lambda}^{H}$ solves a Stefan problem with a rescaled invariant free boundary condition, but with a possibly large diffusion coefficient $\lambda^{-1}$ in the equation, that is,
 \begin{equation}\label{eq:StefanHyper}
        \begin{cases}
         \lambda \partial_{t} u_{\lambda}^{H}-\Delta u_{\lambda}^{H}  = \lambda f(\lambda x, \lambda t)& \text{in} \,\, \left(B_1 \times (-1,0]  \right) \cap \{ u_{\lambda}^{H} > 0  \},\\
       \partial_{t} u_{\lambda}^{H} = \nor{\nabla u_{\lambda}^{H}}^2& \text{on} \,\, \left(B_1 \times (-1,0]\right) \cap \partial\{u_{\lambda}^{H} > 0  \}.
    \end{cases}
    \end{equation}
    Note that, letting $\lambda \to 0^+$, a solution of \eqref{eq:StefanHyper} formally solves the Hele-Shaw equation for which information about time regularity is lost. We point out that for this problem, recently, in \cite{kim24}, it has been shown that flat free boundaries are $C^{1,\alpha}$ at each time possibly admitting the presence of a nonnegative source and a drift term. 
    
On the other hand, if we consider the parabolic rescaling,
\begin{equation*}
    (x,t) \mapsto u_{\lambda}^{P}(x,t) \coloneqq \frac{u(\lambda x, \lambda^2 t )}{\lambda}, \quad (x,t) \in B_1 \times (-\lambda^{-1},0],
\end{equation*}
then we get an equation of the same form, but a different free boundary condition depending on the coefficient $\lambda$ appears. More precisely, it holds

 \begin{equation*}
        \begin{cases}
      \partial_{t} u_{\lambda}^{P} - \Delta u_{\lambda}^{P}  = \lambda f(\lambda x, \lambda^2 t)& \text{in} \,\, \left(B_1 \times (-\lambda^{-1},0]  \right) \cap \{ u_{\lambda}^{P} > 0  \},\\
       \partial_{t} u_{\lambda}^{P} = \lambda \nor{\nabla u_{\lambda}^{P}}^2& \text{on} \,\, \left(B_1 \times (-\lambda^{-1},0]\right) \cap \partial\{u_{\lambda}^{P} > 0  \}.
    \end{cases}
    \end{equation*}
    Thus, the limiting problem leads to no motion of the free boundary, and therefore, no regularization is expected in this case.
    These two alternative representations of the problem highlight the main difference with respect to the elliptic situation.



    \section{The nonlinear problem arising from the hodograph transform} \label{Section:hodograph}

In this section, we discuss the application of the hodograph transform introduced in \cite{DFS23} to our problem.
 To avoid confusion, we will use $u_{S}$ to denote a solution of the Stefan problem \eqref{eq:StefanNH} and $u_H$ for its hodograph transform.\\
First, we look at the graph of $u_S$ in $\mathbb{R}^{n+2}$
\begin{equation} \label{graph1}
    \Gamma := \{ (x,x_{n+1},t) | \; x_{n+1}= u_S(x_1, x_2, \dots, x_n,t)  \}
\end{equation}
    as the graph of a possibly multi-valued function $u_H$ with respect to the $x_n$-coordinate
    \begin{equation*} 
    \Gamma := \{ (x,x_{n+1},t) | \; x_n= u_H(x_1, x_2, \dots, x_{n-1}, x_{n+1}, t)  \}.
\end{equation*}
We say that $u_H$ is \textit{graphical} with respect to the direction $e_{n+1}$ since its graph can be expressed by \eqref{graph1} where $u_S$ is a single-valued function.  
If we denote the coordinates $(x_1, x_2, \dots, x_{n-1}, x_{n+1})$ by $(y_1, \dots, y_n) $, then we can have a formal expression of $\nabla u_S$ in terms of $\nabla u_H$ for a point on the graph $\Gamma$.
Precisely, we can rewrite
\begin{equation*}
    x_{n+1}= u_S(x_1, x_2, \dots, x_n,t)= u_S(y', u_H(y,t),t).
\end{equation*}
Hence, differentiation in $x_{n+1}$ gives
\begin{equation*}
    1 = \partial_{x_{n+1}} u_S = \partial_{x_n} u_S \cdot \partial_{x_{n+1}} u_H,
\end{equation*}
which implies 
\begin{equation*} \displaystyle
     \partial_{x_n} u_S = \frac{1}{\partial_{x_{n+1}} u_H}.
\end{equation*}
On the other hand, for $x_j$, $j=1,\dots, n-1$, and $t$, we have, using the previous equality,
\begin{align*}
    &0 = \partial_{x_j} u_S + \partial_{x_n} u_S \cdot \partial_{x_j} u_H  \Longrightarrow  \partial_{x_j} u_S = - \frac{\partial_{x_j} u_H}{\partial_{x_{n+1}} u_H },\\
    &{0 = \partial_{t} u_S + \partial_{x_n} u_S \cdot \partial_{t} u_H  \Longrightarrow  \partial_{t} u_S = - \frac{\partial_{t} u_H}{\partial_{x_{n+1}} u_H },}
\end{align*}
and thus
\begin{align*}
    &\nabla u_S = - \frac{1}{\partial_{y_n}u_H}\left( \nabla_{y'} u_H, -1 \right), \quad \partial_{t} u_S = - \frac{\partial_{t}u_H}{\partial_{y_n} u_H}, \\
    &D^{2}u_S= -\frac{1}{\partial_{y_n}u_H} {(A(\nabla u_H))^T} \, D^2 u_H \, A(\nabla u_H),
\end{align*}
where $A(\nabla u_H)$ is a square matrix which coincides with the identity except for the $n$-th row in which the entries are given by the right hand side of $\nabla u_S.$ Then, we can transform the Stefan problem \eqref{eq:StefanNH} to a quasilinear parabolic equation with oblique derivative boundary condition of the type
\begin{equation}\label{eq:HodographStefan}
        \begin{cases}
        \partial_{t} u_H - \tr( \Bar{A}(\nabla u_H) D^2 u_H) + \partial_{y_n}u_H \,  f(y',u_H,t)=0 & \text{in} \,\, \{ y_n > 0 \},\\
       \partial_{t} u_H = g(\nabla  u_H)& \text{on} \,\,  \,\, \{ y_n = 0 \},
    \end{cases}
    \end{equation}
    where, taking $p=(p',p_n) \in \R^n,$ 
    \begin{equation} \label{eq:g}
        g(p)= - \frac{1}{p_n} \left( 1+ \nor{p'}^2 \right)  
    \end{equation} and
    \[
\Bar{A}(p)= \begin{pNiceArray}{c|c}
\text{ \begin{Large} $I_{n-1}$ \end{Large}} &   \displaystyle -\frac{p'}{p_n}\\
  \hline
 \displaystyle -\frac{p'}{p_n} &  \displaystyle \frac{1}{p_n^2} \left( 1+\nor{p'}^2 \right)
\end{pNiceArray},
\]
with eigenvalues 
    \begin{equation*}
       \spec(\Bar{A}) = \left\{1, \frac{1  + \nor{p}^2 - \sqrt{\left(1 + \nor{p}^2\right)^2 - 4p_n^2}}{2p_n^2},\frac{1  + \nor{p}^2 +
       \sqrt{\left(1 + \nor{p}^2 \right)^2 - 4p_n^2}}{2p_n^2}\right\}.
    \end{equation*}
     The matrix $\Bar{A}$ is positive definite and $\partial_{p_n} g(p) > 0 $ as long as $p_n \neq 0$. Define 
    \begin{equation*}
       \mathcal{R}_K:= B_{K}(0) \cap \{ p_n  \geq K^{-1}\} \subset \mathbb{R}^{n},
    \end{equation*}
    and notice that if $p\in \mathcal{R}_K$ , then $\partial_{{ p_n}} g(p) > 0 $. 
    If $K > 0$  is chosen large enough, then we may assume that $\tr(\Bar{A}(p)M) $ is uniformly elliptic in $M$ with ellipticity constants $K$ and $K^{-1}$ in $\mathcal{R}_K$.
    Furthermore, in this set, we can suppose that
    \begin{equation} \label{eq:boundg}
        \nor{\nabla_{p} \tr(\Bar{A}(p)M)} \leq K \nor{M}, \quad \norm{g}_{C^1} \leq K, \quad \partial_{p_n} g(p) \geq K^{-1}.
    \end{equation}

The hodograph transform is very useful because it flattens the free boundary, and thus, the problem is transformed from a free boundary problem into one with a fixed boundary but with possible multi-valued solutions. This issue of multi-valued functions becomes the main subtlety with the transformed equation \eqref{eq:HodographStefan}.
Indeed, one must adapt the definition of viscosity solution to include possibly multi-valued functions as solutions of the following fully nonlinear problem, recall \eqref{eq:C_r-F_r},
\begin{equation}\label{eq:FullynonlinearProb}
\begin{cases}\partial_t u = F(x,t,u,\nabla u, D^2u)  &\text{in $\mathcal C_{\lambda}$ }\\
\partial_t u =g( \nabla u)& \text{on $\mathcal F_\lambda$.} 
\end{cases}
\end{equation}
\begin{Definition}
    We say that a  continuous function  $u : \overline{\mathcal{C}_\lambda} \rightarrow \R$ is a \textit{viscosity subsolution} to \eqref{eq:FullynonlinearProb} if $u$ cannot be touched by above at points in $\mathcal{C}_\lambda \cup \mathcal{F}_\lambda$ by a strict $C^2$-supersolution $\varphi$ of \eqref{eq:FullynonlinearProb}. More precisely, we require that there do not exist points $(x_0,t_0) \in \mathcal{C}_\lambda \cup \mathcal{F}_\lambda $ and test functions $\varphi \in C^2( \mathcal{P}_r(x_0,t_0))$ satisfying 
\begin{equation*}
\begin{cases}    
\partial_t \varphi > F(x,t, \varphi,\nabla \varphi, D^2\varphi)  &\text{in $ \mathcal{P}_r(x_0,t_0),$}\\
\partial_t \varphi >g( \nabla \varphi)& \text{on $\mathcal{F}_\lambda \cap  \left( \mathcal{P}_r(x_0,t_0)\right)$,}
\end{cases}
\end{equation*}
such that 
\begin{equation} \label{eq:TestFullynonlinear2}
    u(x_0,t_0)= \varphi(x_0,t_0), \; u \leq \varphi \; \text{in  $ \mathcal{P}_r(x_0,t_0).$}
\end{equation}
Similarly, we can define \textit{viscosity supersolutions} and \textit{viscosity solutions} to \eqref{eq:FullynonlinearProb}.
\end{Definition}

This definition can be generalized to multi-valued functions requiring the comparison with a (single-valued) test function $\varphi$ to hold for all possible values.
\begin{Definition} \label{def:multivalue_sol}
    Let $u : \overline{\mathcal{C}_\lambda} \rightarrow \R$ be a multi-valued function with compact graph in $\R^{n+2}$. We say that $u$ is a viscosity subsolution of \eqref{eq:FullynonlinearProb} if the definition above holds and \eqref{eq:TestFullynonlinear2} is understood as $\varphi(x_0,t_0) \in u(x_0,t_0)$ and $u(x,t) \leq \varphi(x,t)$ for all possible values of $u$ at $(x,t)$ and all $(x,t) \in  \mathcal{P}_r(x_0,t_0). $
\end{Definition}

\begin{Remark} \label{Rmk:flat}
    Using the  flatness assumption \eqref{Flatness}, we can localize the set where $u_H$ is possibly multi-valued. We will prove that $u_H$ is single-valued in the set $\{y_n \geq C_0\varepsilon_1\lambda\}$.  Let $t\in [t_0-\lambda^2,t_0+\lambda^2]$. Then, from the bounds on  $|a_n'|$ and $|b'|$ in the assumptions of Theorem \ref{Theorem:flatsolution}  and \eqref{eq:relaxed_hp}, we get
    \begin{equation*}
            {\nor{a_{n}(t)-a_{n}(t_0)} \leq c_0,
            \qquad\nor{b(t)-b(t_0)} \leq  C}\lambda^{2}.
    \end{equation*}
Without loss of generality, we can assume {that $\lambda, c_0 \leq \varepsilon_1$,} for which we have, according to the assumptions of Theorem \ref{Theorem:flatsolution} again, with $|b(t_0)|\le \lambda/2$, 
\begin{align*}
    &\nor{u_S-a_n(t_0) (x_n-b(t_0))^{+}} \leq \,\nor{u_S - a_n(t)(x_n - b(t))^+} + \nor{a_n(t_0)} \nor{(x_n -b(t))^{+} -(x_n -b(t_0))^{+}  } \\
    &+\,\nor{a_n(t)-a_n(t_0)} \nor{(x_n-b(t))^{+}}\leq\, C \lambda \varepsilon_1 + C \nor{b(t) - b(t_0)} + c_0 \nor{x_n-b(t)}\\ &
    \leq  C\, \lambda \varepsilon_1 +  C^2 \lambda^2 +  3c_0 \lambda \leq C_1 \varepsilon_1 \lambda\quad\mbox{in } \mathcal{P}_\lambda:={Q_{\lambda}} \times [t_0-\lambda^2, t_0 + \lambda^2].
\end{align*}

Let now $(x_0,t_0)$ be such that $x_0 \in Q_{\lambda}, t_0 > -c\lambda$, and $C_0 \varepsilon_1 \lambda \leq u_S(x_0,t_0) \leq c\lambda$, for some $C_0$ to be chosen large. Note that $u_S-a_n(t_0)(x_n - b(t_0))^+$ still solves the heat equation with the same source term in cylinders away from the free boundary.\\ We want to estimate $d:=\text{dist}((x_0,t_0),\partial \{u_S>0\})$. Since 
\begin{align*}
    C_0\varepsilon_1 \lambda\leq u_S(x_0,t_0)\leq C_1 \varepsilon_1 \lambda+a_n(t_0)((x_0)_n-b(t_0))^+,
\end{align*}
we obtain
\begin{align*}
    (x_0)_n\geq b(t_0)+(C_0-C_1)\lambda \varepsilon_1 a_n^{-1}(t_0),
\end{align*}
{provided that $C_0>C_1$}.
For the same time $t_0$, let $\bar x$ be the closest point to $x_0$ on the free boundary. We have $u_S(\bar x, t_0)=0$. If $\bar x \not \in Q_\lambda(x_0)$, then it immediately follows because $d\geq \lambda$. Otherwise, it holds, from above,
\[
    a_n(t_0)(\bar x_n-b(t_0))^+\leq {C_1} \varepsilon_1 \lambda{,}
\]
{which gives}
\[
    \bar x_n\leq b(t_0)+a_n^{-1}(t_0){C_1}\varepsilon_1\lambda.
\]
Therefore,
\[
    d\geq |(x_0)_n-\bar x_n|\geq (C_0-2{C_1})\varepsilon_1\lambda a_n^{-1}(t_0).
\]
Combining this with the interior gradient estimates, see, for instance, \cite{LiebermanParabolic}, we get that there exists $C_e${, only depending} on $n$, such that, {recalling that $\| f\|_{L^\infty(\mathcal{P}_\lambda)}\leq \varepsilon_1^2$,}
\begin{align*}
    &\nor{\nabla u_S (x_0,t_0) - a_n(t_0)e_n} \leq \, C_e d^{-1}\left(\norm{u_S-a_n(t_0)(x_n - b(t_0))^+}_{L^{\infty}(\mathcal{P}_\lambda)}+ \lambda \norm{f}_{L^{\infty}(\mathcal{P}_\lambda)} \right)\\
    &\leq \, C_e(C_0-2C_1)^{-1}(\varepsilon_1\lambda)^{-1}a_n(t_0)\left(C_1\varepsilon_1\lambda+\lambda \varepsilon_1^2\right) =\, C_e\frac{{C_1}+\varepsilon_1}{C_0-2{C_1}}a_n(t_0){.}
\end{align*}
Taking $C_0=(4C_e+2){C_1}${, this} yields
\begin{align*}
    \nor{\nabla u_S (x_0,t_0) - a_n(t_0)e_n}\leq \frac{1}{2}a_n(t_0),
\end{align*}
which implies $|\nabla u_S(x_0,t_0)|\geq a_n(t_0)/2\geq (2K)^{-1}${. Hence,} by the inverse function theorem, the hodograph transform $u_H$ is well defined at such points. Actually, we get $\nabla u_S(x_0,t_0)\in \mathcal{R}_{2K}${,} which, in terms of the hodograph transform, still reads $\nabla u_H(x_0,t_0)\in \mathcal{R}_{CK}$, where $C>0$ is a large universal constant, which we will be omitted in the following renaming $CK$ with $K$.

\end{Remark}

\begin{Remark}\label{Rmk:flatimpliesreg}
     The main hypothesis of Theorem \ref{Theorem:flatsolution} can be written in terms of the hodograph transform as
    \begin{align*}
        &\nor{u_H - (\bar a_n(t)y_n+\bar b(t))} \leq C\varepsilon_1 {\bar \lambda}\quad   {\text{in }} \mathcal{C}_{\bar \lambda},  \\
        &\bar b'(t)=g(\bar a_n(t) e_n), \quad  K^{-1} \leq \bar a_n \leq K, \quad \nor{\bar a_n'} \leq c_1 \bar \lambda^{-2},
    \end{align*}
    with $\bar \lambda\leq {\bar\lambda_1}$ for some $\bar \lambda_1>0$ small. 
    
    Therefore, we can reduce the initial problem of obtaining regularity of a flat free boundary for $u_S$ to proving an improvement of flatness for the hodograph transform $u_H$.
\end{Remark}

For simplicity, from now on, we will refer to the transform function $u_H$ as $u$ and the solution of \eqref{eq:StefanNH} will be denoted by $u_S$. We emphasize that $u$ carries many properties from being the transform function of $u_S$. Moreover, we drop the bar from $\bar A, \bar \lambda, \bar a, \bar b$ and {use} $x$ rather than $y$.

\subsection{The inhomogeneous $H$-property}\label{The inhomogeneousHproperty} In this subsection, we explore an essential Harnack-type property of solutions {to} \eqref{eq:HodographStefan}. For this purpose, we introduce the function
\begin{equation*}
    l_{a,b}(x):= a \cdot x + b. 
\end{equation*}
Here, $a\in \mathbb{R}^{n}$, $\nor{a} \leq K$, and $b \in \mathbb{R}$. In the following, $l_{a,b}$ will denote a more general function, linear in $x$. Then, $\omega =u- l_{a,b}$ solves
\begin{equation*}
     \partial_{t} \omega - \tr( \Bar{A}(\nabla \omega + a) D^2 \omega) + (\partial_{y_n}\omega + a_n) \,  f(y',\omega + a \cdot x + b,t)=0.
\end{equation*}
By uniform ellipticity, we have
\begin{equation*}
\begin{aligned}
    \partial_{t} \omega -\Mp(D^2 \omega) + \partial_{y_n}\omega f + a_n f \leq 0,
\end{aligned}
\end{equation*} 
and thus,
\begin{align}\label{eq:M+omega}
     \partial_{t} \omega -\Mp(D^2 \omega) - \nor{\nabla \omega} \norm{f}_{L^{\infty}} - a_n \|f\|_{\infty} \leq 0.
\end{align}
Analogously for $\Mm$,
\begin{equation}\label{eq:M-omega}
    \partial_{t} \omega - \Mm(D^2 \omega) + \nor{\nabla \omega} \norm{f}_{L^{\infty}} + a_n \|f\|_{\infty} \geq 0.
\end{equation}

Note that, for every $Q_r(x_0)\subset \mathcal{C}_\lambda$ with $r>2C_0\varepsilon_1\lambda${, it holds} that $u$ (and {hence,} $\omega$) is single-valued in $Q_{r/2}(x_0)\times [t_0-r^2,t_0+r^2]$ by Remark \ref{Rmk:flat}. {Consequently,} by the Harnack inequality for the extremal equations \eqref{eq:M+omega} and \eqref{eq:M-omega} obtained in  \cite[Corollary 4.6]{koike2019weak}, we have{, recalling that $K^{-1} \leq  a_n \leq K$,}
\begin{equation*}
    \sup_{\substack{Q_{r/2}(x_0) \times [t_0-r^2, t_0]}} \omega \leq C_H \left( \inf_{\substack{Q_{r/2}(x_0) \times [t_0+\frac{r^2}{2}, t_0+r^2]}} \omega + C_n a_n \norm{f}_{L^{\infty}}\right),
\end{equation*}
where $C_H=C_H(n,K,\norm{f}_{L^{\infty}})$ and $C_n$ is a dimensional constant, if $$\omega=u - l_{a,b} \geq 0 \qquad \text{ in } Q_{r}(x_0) \times [t_0-r^2, t_0+r^2], \quad r \geq C \varepsilon_1 \lambda.$$ {Therefore, if} 
\begin{equation*}
    {\omega(x_0,t_0)=(u-l_{a,b})(x_0,t_0)} \geq \mu, \qquad \text{for some } \mu \geq 0,
\end{equation*}
then {we infer}
\begin{equation*}
    {u-l_{a,b}} \geq c_1 \mu - K \norm{f}_{L^{\infty}} \geq c_1\mu - K\varepsilon_1^{2}\ge c\mu\quad \mbox{in } Q_{r/2}({x_0}) \times {\Big[t_0+\frac{r^2}{2}, t_0+r^2\Big]},
\end{equation*}
up to renaming $K$ and as long as $\varepsilon_1$ is sufficiently small. An analogous statement holds for ${\omega\le 0}$.

Now, we are ready to state the $H(\sigma)$-property for a general $\sigma >0$ small, which we have just proved for $\sigma=c\varepsilon_1$.
\begin{Definition}[$H(\sigma)$-property]\label{Def:H-prop}
    Given a {small constant $\sigma>0$}, we say that \textit{$u$ has {the} property $H(\sigma)$ in $\mathcal{C}_\lambda$} if {when}
    \begin{equation*}
        u - l_{a,b} \geq 0 \qquad \text{ in } Q_{r}(x_0) \times [t_0-r^2, t_0+r^2] \subset \mathcal{C}_{\lambda},
    \end{equation*}
    for any $r \in [\sigma \lambda, \lambda]$, with 
    \begin{equation*}
    (u - l_{a,b})(x_0,t_0)\geq \mu, \qquad \text{for some } \mu \geq 0,
\end{equation*}
{there} exists a universal $c>0$ such that \begin{equation*}
    u-l_{a,b} \geq c\mu \quad \text{ in } Q_{\frac{r}{2}}(x_0) \times \Big[t_0+\frac{r^2}{2}, t_0+r^2\Big].
\end{equation*}
\end{Definition}
Property $H(\sigma)$ for all $\sigma >0$ is a consequence of the parabolic Harnack inequality in \cite{koike2019weak} when $u$ is a single-valued viscosity solution of \eqref{eq:HodographStefan}. Property $H(\sigma)$ for a multi-valued solution of \eqref{eq:HodographStefan} roughly states that u behaves
 as a single-valued function from scale $\lambda$ up to scale $\sigma \lambda$. 
 Indeed, we will show that property $H(\sigma)$ holds for multi-valued solutions that are graphical in the $e_n$-direction and are well approximated by functions of the form $l_{a,b}$. These properties of a solution to \eqref{eq:HodographStefan} stem from the fact that it is obtained via a transformation of a solution to \eqref{eq:StefanNH} and the flatness assumption.

\subsection{Statement and discussion of our main result}

Let $\ell_{a,b}$ be the linear approximation of $u$ in $x$ for each $t$ fixed, defined as 
\begin{equation}\label{eq:linear-approx}
    \ell_{a,b}(x,t) := a(t) \cdot x + b(t),
\end{equation}
where
\begin{equation}\label{eq:condit-linear-approx}
    a(t):= (a_1, \dots, a_{n-1}, a_n(t))= (\hat{a},a_n(t)), \; \hat{a} \in \mathbb{R}^{n-1}.
\end{equation}
In order to satisfy boundary condition of \eqref{eq:QuasiStefanNH}, we further ask
\begin{equation} \label{eq:b1}
    b'(t) = g(a(t)).
\end{equation}
{We are now} able to state our main result concerning the improvement of flatness for a solution to \eqref{eq:QuasiStefanNH}.

\begin{Proposition}[Improvement of flatness] \label{Prop:ImprovFlat}
    Fix $K > 0$ large. Let $u$ be a graphical in the direction $e_n$, possibly multi-valued, viscosity solution to \eqref{eq:QuasiStefanNH}
    with the assumption $\norm{f}_{L^{\infty}(\Bar{\mathcal{C}}_\lambda)} \leq {\varepsilon^2}$. {Suppose that $u$ satisfies} the property $H(\varepsilon^{1/2})$ and 
    \begin{equation} \label{Hp:assum-flat}
            \nor{u-\ell_{a,b}} \leq \varepsilon \lambda\quad\text{in } \bar{\mathcal{C}}_{\lambda},
    \end{equation}
    with 
    \begin{equation}\label{Hp:condit-assum-flat}
    	b'(t)=g(a(t)),\quad a(t) \in \mathcal{R}_{K},\quad \nor{a_n'(t)} \leq \delta \varepsilon \lambda^{-2},
    \end{equation}
    where
    \begin{equation*}
        \varepsilon \leq \varepsilon_{1}, \quad \lambda \leq \lambda_{1}, \quad \lambda \leq \delta \varepsilon.
    \end{equation*}
    Then, there exists $\ell_{\Tilde{a},\Tilde{b}}$ such that
    \begin{equation*}
        \nor{u- \ell_{\Tilde{a},\Tilde{b}}} \leq \frac{\varepsilon}{2} \tau \lambda \; \text{ in } \Bar{\mathcal{C}}_{\tau \lambda }, \quad \Tilde{b}'(t)=g(\Tilde{a}(t)),
    \end{equation*}
    with
    \begin{equation*}
        \nor{a(t) - \Tilde{a}(t)} \leq C \varepsilon, \quad \nor{\Tilde{a}_n'(t)} \leq \frac{\delta \varepsilon}{2} (\tau \lambda)^{-2}.
    \end{equation*}
    Here, the constants $\varepsilon_{1}, \lambda_{1}, \delta, \tau > 0$ small, and $C$ large only depend on $n$ and $K$.
    \end{Proposition}
     
\begin{Remark}\label{Rmk:H-rescale}
 We now justify the assumptions of Proposition \ref{Prop:ImprovFlat}. Since we apply it to the hodograph transform of a solution to the original Stefan problem, we already know that it is graphical with respect to the $e_n$-direction. Furthermore, from \eqref{Hp:assum-flat} and using the bounds on $|a_n'|$ and $|b'|$ as in Remark \ref{Rmk:flat}, we can show that in each parabolic cylinder of size $\lambda$, $u$ is well approximated by a time-independent linear function, namely, 
\begin{equation}\label{eq:approx_hodog_linear_const_time}
	|u - (a(t_0)\cdot x +b(t_0))| \leq 2\varepsilon \lambda \qquad \text{in } Q_{\lambda}^+ \times [t_0-\lambda^2,t_0+\lambda^2].
\end{equation}
Now, by virtue of the discussion in Subsection \ref{The inhomogeneousHproperty}, it follows that $u$ satisfies the property $H(c\varepsilon_1)$. Taking $\varepsilon_1$ small enough, we get that $u$ also satisfies $H(\varepsilon^{1/2})$ as well.
\end{Remark}

\vspace{5pt}
The proof of Proposition \ref{Prop:ImprovFlat} will be presented in Section \ref{Section:C1a}.
Here, we start by defining a useful function, expressing 
\begin{equation}\label{eq:w_error}
    u(x,t) = \ell_{a,b}(x,t) + \varepsilon \lambda w\left( \frac{x}{\lambda}, \frac{t}{\lambda}\right), \quad (x,t) \in \mathcal{C}_{\lambda},
\end{equation} 
where $w(x,t)$ is the rescaled error function, defined in $\mathcal{C}_1,$ possibly multi-valued near $\{ x_n = 0 \}$, and satisfying $\nor{w} \leq 1$ in $\mathcal{C}_1$, {in view of \eqref{Hp:assum-flat}}.\\
Exploiting \eqref{eq:w_error}, \eqref{eq:linear-approx}, and \eqref{eq:condit-linear-approx}, \eqref{eq:QuasiStefanNH} becomes
    \begin{equation}\label{eq:LinStefanNH}
        \begin{cases}
        \lambda^2 \varepsilon^{-1} a_n'(\lambda t) x_n + \lambda \varepsilon^{-1} b'(\lambda t) + \lambda \partial_t w(x,t) - \tr(A(a(\lambda t) + \varepsilon \nabla w)) D^2 w ) +\\ \left( \lambda \partial_{x_n}w+\lambda \varepsilon^{-1} a_n(\lambda t) \right) f(x',\ell_{a,b}+\varepsilon\lambda w,t) =0  & \text{in } \mathcal{C}_1,\\
       \varepsilon^{-1} b'(\lambda t) + \partial_{t} w = \varepsilon^{-1}  g(a(\lambda t) + \varepsilon \nabla w)& \text{on } \mathcal{F}_1.
    \end{cases}
    \end{equation}
  {From \eqref{eq:g}, \eqref{eq:b1}, and again \eqref{eq:condit-linear-approx},} we can write the boundary condition in a convenient way as
    \begin{equation*}
   \partial_{t} w =  \displaystyle\frac{(1 + \nor{\hat{a}}^2) \partial_{x_n}w}{a_{n}(\lambda t) (a_{n}(\lambda t) + \varepsilon \partial_{x_n}w)} - \frac{\varepsilon \nor{\nabla_{x'}w}^2+ 2 \ps{\hat{a}}{\nabla_{x'}w}}{a_{n}(\lambda t) + \varepsilon \partial_{x_n} w} \quad \text{on} \,\,  \,\, \mathcal{F}_1.
    \end{equation*}
    Now, letting $\varepsilon \rightarrow 0, \delta \rightarrow 0 $  and using that $\nor{a_n'} \leq \delta \varepsilon \lambda^{-2}$, $a(\lambda t)  \in \mathcal{R}_K$, and $\lambda \varepsilon^{-1} \leq \delta \rightarrow 0$, we formally obtain, in light of \eqref{eq:LinStefanNH}, that our equation gets closer to 
    \begin{equation}\label{eq:LimitStefanNH}
        \begin{cases}
        \lambda \partial_t v = \tr(A_{\lambda}(t) D^2 v )& \text{in} \,\, \mathcal{C}_1,\\
     \partial_t v = \gamma_{\lambda} (t) \cdot \nabla v & \text{on} \,\,  \,\, \mathcal{F}_1,
    \end{cases}
    \end{equation}
    where  
    \begin{equation*}
        A_{\lambda}(t):= A(a(\lambda t)), \quad \gamma_{\lambda} (t) :=  \nabla g (a(\lambda t)).
    \end{equation*}
{According to the fact that $A,g \in C^2(\mathcal{R}_K)$, \eqref{eq:boundg}, and \eqref{Hp:condit-assum-flat}}, we have {the} following conditions on $A_\lambda$ and $\gamma_\lambda$ 
\begin{align*}
    K^{-1}I \leq A_\lambda(t) \leq K I,  \quad
    \nor{A_\lambda'(t)} \leq \lambda^{-1}, \quad  \gamma_\lambda \in \mathcal{R_K},
    \quad
    \nor{\gamma_\lambda'(t)} \leq \lambda^{-1}.
\end{align*}



\section{Regularity result for the error function} \label{Section:propertiesv}
In this section, first we recall some regularity results for solutions to the homogeneous version of \eqref{eq:LimitStefanNH}, then we obtain new comparison results relating this equation to \eqref{eq:LinStefanNH}.
Finally, we prove some regularity results for the error function $w$.

In the following, it will be useful to consider the distance in $\R^{n+1}$
 \begin{align*}
    d_\lambda((x,t),(y,s)):=\min\left\{|x'-y'|+|x_n-y_n|+\lambda^{-\frac{1}{2}} |t-s|^{\frac{1}{2}}, |x'-y'|+|x_n| + |y_n| +|t-s|    \right\},
 \end{align*}
interpolating between the parabolic and standard Euclidean distances depending on the location of the points. We call  $\mathcal{B}_{\lambda,r}$ the ball induced by the metric $d_\lambda$, defined  as in \cite[Section 5.1]{DFS23}.  More precisely, for $x_0=(x_0',(x_0)_n) \in \R^{n-1} \times \R$ and  $t_0 \in \R$,  it holds
	 \begin{equation}\label{eq:balls-d_lambda}
		\mathcal{B}_{\lambda,r}(x_0,t_0):=\begin{cases}
			 Q_r(x_0)\times (t_0-\lambda r^2,  t_0)\quad&\text{if }r<|(x_0)_n|,\\
			 Q_r^+(x_0)\times (t_0-r,  t_0)\quad&\text{if }|(x_0)_n|\le r\le \lambda^{-1}.
		\end{cases}
	\end{equation}
Also, we denote by \[ \norm{v}_{C_{d_\lambda}^{\alpha}}:= \norm{v}_{L^\infty}+  \sup_{(x,t)\neq (y,s)} \frac{|v(x,t)-v(y,s)|}{d_\lambda((x,t),(y,s))^{\alpha}}\]
the H\"older spaces with respect to this metric.

\subsection{Regularity results for the limiting profile} 

The following regularity results can be found in \cite{DFS23}.

\begin{Proposition}[Interior estimates] \label{Prop:Interiorv}
    Let $v$ be a viscosity solution of \eqref{eq:LimitStefanNH} such that \\$\|v\|_{L^\infty} \leq 1.$ Then,
$$| \nabla v|, \, \, \,  |D^2 v| \, \, \leq C \quad \quad \textit{in} \quad \mathcal C_{1/2},$$
and there exists $\ell_{\bar a,\bar b}$  such that for every $\rho \leq 1/2$,
$$|v - \ell_{\bar a, \bar b}|\leq C \rho^{1+\alpha} \quad \quad \text{in $\mathcal C_{\rho}$},$$
with 
$$\bar b'(t)= \gamma_\lambda(t) \cdot \bar a, \quad \quad |\bar a'_n| \leq C \rho^{\alpha-1}\lambda^{-1}, \quad \quad |\bar a| \le C,$$ for some $\alpha$, $C$ universal.

\end{Proposition}

\begin{Proposition}[The Dirichlet problem] \label{Prop:Dirichletv}
    Let $\phi$ be a continuous function on $\p_D \mathcal C_1$. Then, there exists a unique classical solution $v \in C^{2,1} (\mathcal C_1) \cap C(\bar {\mathcal C_1})$ to the Dirichlet problem \eqref{eq:LimitStefanNH} with $v=\phi$ on $ \p_D \mathcal C_1$. Moreover,
$$|\nabla v|, |D^2 v| \le C(\sigma) \|v\|_{L^\infty} \quad \textit{in} \quad \mathcal C_1^\sigma:=\{(x,t):\,d_\lambda((x,t), \p_D \mathcal C_1) \ge \sigma\},$$
and if $\phi$ is $C^\alpha$ with respect to the distance $d_\lambda,$ then $v$ is $C^\alpha$ up to the boundary and 
$$\|v\|_{C^
\alpha_{d_\lambda}} \le C \|\phi\|_{C^\alpha_{d_\lambda}},$$
with $C(\sigma)$, $C$ universal constants, independent of $\lambda$.
\end{Proposition}

\subsection{Comparison principle}
We establish that a solution of \eqref{eq:LinStefanNH} is close to a solution of $\eqref{eq:LimitStefanNH}$. This is a consequence of the comparison principle stated in the following proposition.
\begin{Proposition}\label{Prop:comparison1}
    Let $v \in C^{2}(\Bar{\Omega})${, with $\Omega \subset \mathcal{C}_1$,} satisfy
    \begin{equation*}
        \nor{\nabla v}, \, \nor{D^2 v} \leq M,
    \end{equation*}
    for some large constant $M$, and 
    \begin{equation}\label{eq:CompareLimitStefanNH}
        \begin{cases}
        \lambda \partial_t v \leq \tr(A_{\lambda}(t) D^2 v ) - C \delta& \text{in} \,\, \Omega,\\
     \partial_t v \leq \gamma_{\lambda} (t) \cdot \nabla v - \delta & \text{on} \,\,  \,\, \mathcal{F}_1 \cap \Bar{\Omega},
    \end{cases}
    \end{equation}
    where $A_{\lambda}(t), \gamma_{\lambda}(t)$ as above.
    Then, $v$ is a subsolution of \eqref{eq:LinStefanNH}, as long as $C${, universal,} is sufficiently large and $\varepsilon \leq \varepsilon_1(\delta, M)$. In particular, if 
    \begin{equation*}
        v \leq w \; \text{ on } \overline{\partial \Omega \setminus (\{ t=0\} \cup \{x_n = 0 \})},
    \end{equation*}
    then 
    \begin{equation*}
        v \leq w \text{ in } \Omega.
    \end{equation*}
\end{Proposition}
An analogous result holds in the case of a supersolution, replacing $\leq$ with $\geq$ and the $-$ in \eqref{eq:CompareLimitStefanNH} by $+$. \\
\begin{proof}
    {Notice that, from the assumptions of Proposition \ref{Prop:ImprovFlat} and \eqref{eq:boundg}, we have}
    \begin{equation*}
        \lambda\delta \nor{a_n'(\lambda t)} + \delta\nor{b'(\lambda t)} \leq \delta(\lambda K +  K), 
    \end{equation*}
    \[
        \left| \lambda \partial_{x_n}v+\lambda \varepsilon_1^{-1} a_n(\lambda t) \right|\, |f(x',v,t)|\leq \left( \lambda M+\delta K\right)\|f\|_\infty,
    \]
    and 
    \begin{equation*}
        \left|\tr (A(a(\lambda t) + \varepsilon \nabla v) - A(a(\lambda t)) D^2 v )\right| \leq M K \varepsilon_1 \nor{D^2 v} \leq K \varepsilon_1 M^2,
    \end{equation*}
    since $a \in \mathcal{R}_K$. Therefore, we obtain that $v$ is a solution of the first equation of \eqref{eq:LinStefanNH} if we choose $C > 0$ large. Indeed,
  \begin{align*}
      \lambda \partial_t v & \leq \tr(A(a(\lambda t + \varepsilon_1 \nabla v))D^2 v) + K \varepsilon_1 M^2 - C \delta \leq \tr(A(a(\lambda t + \varepsilon_1 \nabla v))D^2 v) + K \varepsilon_1 M^2 - C \delta\\
      - &(\lambda^2 \varepsilon_1^{-1} a_n'(\lambda t) x_n + \lambda \varepsilon_1^{-1} b'(\lambda t))+ \delta (\lambda K + K) - (\lambda \partial_{x_n} w + \lambda \varepsilon_1^{-1} a_n(\lambda t))f + (\lambda M + \delta K) \norm{f}_{L^{\infty}}.
  \end{align*}
  Then, it is enough to choose $C$ so that
  \begin{equation*}
     C \delta \geq K \varepsilon_1 M^2 + \delta (\lambda K + K) +(\lambda M + \delta K) \norm{f}_{L^{\infty}}.
  \end{equation*}
    The same holds for the boundary condition using 
    \begin{equation*}
        \nor{g(a(\lambda t) + \varepsilon_1 \nabla v) -g(a(\lambda t)) - \varepsilon_1 \nabla g(a(\lambda t)) \cdot \nabla v } \leq C \varepsilon_1^2 M^2.
    \end{equation*}
    The second part of the statement directly descends from the standard comparison principle.
\end{proof}
Thus, it follows  that if the rescaled error $w$ is close to a $C^2$-solution $v$ of \eqref{eq:LimitStefanNH} on the Dirichlet boundary of $\Omega \subset \mathcal{C}_1$, then $v$ and $w$ remain close to each other in the whole domain $\Omega$.
\begin{Corollary} \label{Cor:comparison}
    Let $w$ be a solution of \eqref{eq:LinStefanNH} and $v \in C^{2}$ be a solution of \eqref{eq:LimitStefanNH}  in a domain $\Omega \subset \mathcal{C}_1$, with
    \begin{equation*}
        \nor{\nabla v}, \nor{D^2 v} \leq M.
    \end{equation*}
    If $\varepsilon \leq \varepsilon_1(\delta,M) $ and 
    \begin{equation*}
        \nor{v-w} \leq \sigma \quad \text{ on } \overline{\partial \Omega \setminus (\{ t=0\} \cup \{x_n = 0 \})},
    \end{equation*}
    then \begin{equation*}
        \nor{v-w} \leq \sigma + C\delta \quad \text{in } \Omega. 
    \end{equation*}
\end{Corollary}
\begin{proof}
    This corollary follows from Proposition \ref{Prop:comparison1} if, instead of $v$, we take
    \begin{equation*}
        v \pm \left( C \delta(x_n^2 -t -2)-\sigma \right).
    \end{equation*}
\end{proof}
    
\subsection{The oscillation decay properties of the error function $w$}\label{Section:Osc_w}

Let $w$ be given by \eqref{eq:w_error}. We want to prove the following oscillation property for the error function $w$.

\begin{Proposition} \label{prop:wharnack}
    Assume that $u$ satisfies the hypotheses of Proposition \ref{Prop:ImprovFlat} and let $w$ be given by \eqref{eq:w_error}. Then,
    \[
        \osc_{\mathcal{C}_{1/2}}w\leq 2(1-c),
    \]
    with $c$ small universal, provided that $\delta \leq K^{-1}$ and $\varepsilon \leq c_1\delta$, with $c_1$ universal.
\end{Proposition}
In order to prove Proposition \ref{prop:wharnack}{, we consider the function $\bar w $ defined as
\begin{equation}\label{eq:w-bar}
  \bar w:=w+1+C\delta\left(2+t-x_n^2\right)\geq 0,
\end{equation}
recalling that $|w|\le 1$ in $\mathcal{C}_1$.
From \eqref{eq:LinStefanNH}{, it immediately} follows that $\bar w $ solves}
 \begin{equation} \label{eq:errorshifted}
 \begin{cases}
     \lambda \bar w_t = \tr(A(a(\lambda t) + \varepsilon(\nabla \bar w + 2C\delta x_n e_n))(D^2 \bar w + 2C\delta (e_n \otimes e_n))) - \lambda^2 \varepsilon^{-1} a_n'(\lambda t) x_n + \\
     - \lambda \varepsilon^{-1} b'(\lambda t) - \lambda (\p_{x_n} \bar w + 2C \delta x_n + \varepsilon^{-1} a_n (\lambda t)) f +C \delta \lambda  &\text{in }\mathcal{C}_1, \\
     \bar w_{t} = -  \varepsilon^{-1} g(a(\lambda t)) + \varepsilon^{-1} g(a(\lambda t) + \varepsilon(\nabla \bar w + 2C \delta x_n e_n)) +C \delta & \text{on }\mathcal{F}_1.
     \end{cases}
 \end{equation}
 \begin{Lemma} \label{Lemma:barwremark}
 Let $v \in C^{2}(\mathcal{C}_1)$, with $\nor{\nabla v}, \nor{D^2 v} \leq M$, be a viscosity solution of
  \begin{equation*}
        \begin{cases}
        \lambda \partial_t v \leq \Mm(D^2 v) & \text{in }  \mathcal{C}_1,\\
     \partial_t v \leq K^{-1} v_n^+ - Kv_n^- - K \nor{\nabla_{x'} v}  & \text{on } \mathcal{F}_1.
    \end{cases}
\end{equation*}
Then, $v$ is a subsolution of \eqref{eq:errorshifted} as long as $\delta \leq K^{-1}, \varepsilon \leq c_1 \delta$.
\end{Lemma}
\noindent\textit{Proof of Lemma~\ref{Lemma:barwremark}.~}
Exploiting the properties of Pucci operators,
 \begin{equation*}
         \Mm(D^2 v + 2C\delta e_n \otimes e_n) \geq \Mm(D^2 v) + \Mm(2C\delta e_n \otimes e_n) \geq  \Mm(D^2 v) + \frac{2C\delta}{K}.
 \end{equation*} 
Hence,
\begin{equation*}
    \begin{aligned}
        \lambda \p_t v &\leq \Mm(D^2v  + 2C\delta e_n \otimes e_n) -\frac{2C\delta}{K}\\ &\leq \tr(A(a(\lambda t) + \varepsilon( \nabla v + 2C\delta x_ne_n)) (D^2 v +  2C\delta  e_n \otimes e_n ) ) - \frac{2C\delta}{K}.
    \end{aligned}
\end{equation*}
Define the following quantity to lighten the notation
\begin{equation*}
    H(v,\p_{x_n}v) :=  -\lambda^2 \varepsilon^{-1} a_n'(\lambda t)x_n - \lambda \varepsilon^{-1} b'(\lambda t)- \lambda (\p_{x_n} v + 2C \delta x_n  + \varepsilon^{-1} a_n( \lambda  t) ) f +C\delta \lambda.
\end{equation*}
Then, as $\lambda \leq \lambda_1 \leq 1$, we have{, again from the assumptions of Proposition \ref{Prop:ImprovFlat},} that $ H(v,\p_{x_n} v ) \geq - (2K + 1 + M + C\lambda) \delta $, which implies that 
\begin{align*}
     \lambda \p_t v  &\leq \tr(A(a(\lambda t) + \varepsilon( \nabla v + 2C\delta x_ne_n)) (D^2 v + 2C \delta e_n \otimes e_n ) )  + H(v,\p_{x_n} v)  - \frac{2C\delta}{K}\\
     &+ (2K + 1 + M +  C \lambda) \delta.
\end{align*}
Choosing $\delta \leq K^{-1}$ and $C$ large enough leads to
\begin{equation*}
     \lambda \p_t v  \leq \tr(A(a(\lambda t) + \varepsilon( \nabla v + 2C\delta x_ne_n)) (D^2 v + 2C \delta e_n \otimes e_n ) )  + H(v,\p_{x_n} v) .
\end{equation*}

On the other hand, for the boundary condition, it holds, recalling that $a(t) \in \mathcal{R}_{K}$ and $g\in C^2(\mathcal{R}_{K})$,
\begin{equation*}
    g(a(\lambda t) + \varepsilon(\nabla v + 2 C \delta x_ne_n )) \geq g (a(\lambda t))+ \nabla g (a(\lambda t)) \cdot \varepsilon (\nabla v + 2C \delta x_ne_n) - \varepsilon^2 C_1 \nor{\nabla v + 2C \delta x_n e_n}^2.
\end{equation*}
So, we achieve, according to \eqref{eq:boundg},
\begin{align*}
    &\nabla g(a(\lambda t)) \cdot \nabla v \\
    & \leq\,\varepsilon^{-1} \left( g(a(\lambda t) + \varepsilon(\nabla v + 2 C \delta x_ne_n )) - g(a(\lambda t))\right) - 2C \delta x_n g_n(a(\lambda t)) + 2C_1 \varepsilon \left( \nor{\nabla v}^2 + 4 C^2 \delta^2\right)   \\ 
    &\leq \, \varepsilon^{-1}g(a(\lambda t) + \varepsilon(\nabla v + 2 C \delta x_ne_n )) - \varepsilon^{-1} g(a(\lambda t)) + \delta \left(  2KC  + C_2\varepsilon  \delta^{-1} M^2 + 4 \varepsilon C_2\delta \right),
\end{align*}
which yields, assuming $\varepsilon \leq c_1\delta$, being $K^{-1}\le g_n\le K$ in $\mathcal{R}_{K}$,
\begin{align*}
    \p_{t} v \leq \,&K^{-1} v_n^+ - Kv_n^- - K \nor{\nabla_{x'} v} \leq \nabla g ( a (\lambda t)) \cdot \nabla v \\ 
    \leq\,& \varepsilon^{-1}g(a(\lambda t) + \varepsilon(\nabla v + 2 C \delta x_ne_n )) - \varepsilon^{-1} g(a(\lambda t)) + C\delta, 
\end{align*}
where $C=C(M,K,n)$.  
\hfill $\square$

Now, we state an alternative version of \cite[Lemma 6.2]{DFS23} involving $\bar w$, which describes its evolution in time when compared to an explicit barrier. We start by introducing such barrier. Consider a smooth domain $\Omega$ in $\mathbb{R}^n$, $n \geq 2,$ such that
$$\bar Q^+_{3/4} \subset \bar \Omega \subset \bar Q^+_{7/8},$$
and call
$$T:= \{x_n=0\} \cap Q_{3/4} \subset \partial\Omega.$$ 
Define $\eta(x')$ as a standard bump function supported on $Q'_{5/8}$ and equal to 1 on $Q'_{1/2}$ (here the prime denotes cubes in $\mathbb{R}^{n-1}$).
Also, let $\phi$ satisfy
$$\mathcal M^-_{K}(D^2\phi)= 0 \quad \text{in $\Omega$},$$
$$\phi= 0 \quad \text{on $\partial \Omega \setminus T$}, \quad \phi=\eta \quad \text{on $T$},$$ and notice that $0 \leq \phi \leq 1$, $\phi \geq c$ on $Q_{1/2}^+$, and by the Hopf lemma, $\phi_n>0$ on $\{x_n=0\} \cap \{\phi=0\}$. Furthermore, since $\Mm$ is concave, in view of \cite[Chapter $9$]{CC}, we have that $\phi \in C^{2,\alpha}(\overline{\Omega})$ with universal estimates. 

\begin{Lemma} \label{Lm:Barrier}
\label{phi} Let $\bar w \geq 0$, defined as in \eqref{eq:w-bar}, solve \eqref{eq:errorshifted}
in the viscosity sense. If for some $t_0 \in (-1,0]$,
$$\bar w(x, t_0) \geq s_0 \, \, \phi(x) \quad \text{in $Q^+_1$}, \quad s_0 \geq 0,$$
then
$$\bar w(x,t) \geq s(t) \, \phi(x) \quad \text{in $Q^+_1 \times [t_0, 0],$}$$
with
$$s'(t)= - C_0 s(t), \quad s(t_0)=s_0, \quad \quad \textit{$C_0$ large universal}.$$
Moreover, if $s_0 \le c_0$ with $c_0$ small universal, and
\begin{equation*}
	\bar w \left(\frac 1 2 e_n, t_0+\lambda/4 \right) \geq \frac {1}{2},\quad {\lambda \text{ small universal}},
\end{equation*}
then 
$$\bar w(x,t_0+\lambda) \geq (s_0 + c_0\lambda) \phi(x) \quad{\text{in }Q^+_1}.$$
\end{Lemma}
Note that $\bar w$ is possibly multi-valued near $\mathcal{F}_1$. Therefore, the inequalities hold for every possible value. The proof of Lemma \ref{Lm:Barrier} is identical to \cite[Lemma 6.2]{DFS23} since it still satisfies the interior Harnack inequality, given by the $H(c'')$-property of $w$ for some $c''$ small universal constant, after a rescaling in time of size $\lambda$. This fact, together with Lemma \ref{Lemma:barwremark}, allows us to perform the comparison with the same explicit barriers constructed in the proof given by the authors. \\
We are now ready to provide a proof for Proposition \ref{prop:wharnack}. \newline

\noindent\textit{Proof of Proposition~\ref{prop:wharnack}.~}  We can assume that $w(e_n /2 , t_k + \lambda/4) \geq 0$ for more than half of the values $t_k=-1+k\lambda \in [-1,-1/2]$ because if not, we consider $-w$ satisfying an equation for which it is possible to replicate Lemma \ref{Lm:Barrier}. Therefore, it holds, according to \eqref{eq:w-bar},
 \begin{equation} \label{eq:morethanhalf}
    \bar w(e_n /2 , t_k + \lambda/4) \geq 1,\quad \textit{ for more than half of the values} \quad t_k\in [-1,-1/2].
\end{equation} 

We want to use Lemma \ref{Lm:Barrier} at each time $t_k$, with $s_0=0$ and $s_k=s(t_k)$. We begin by explaining the argument. Suppose we have proved $\bar w\geq s_k\phi(x)$ in $Q_1^+ \times [t_k,0]$. If \eqref{eq:morethanhalf} holds (good case) and $s_k \leq c_0$, then we use the second part of Lemma \ref{Lm:Barrier} to obtain the improvement
\[
\bar w (x,t_{k+1}) \geq (s_k+c_0\lambda)\phi(x)=s_{k+1}\phi(x) \qquad \textit{in }Q_1^+.
\]
On the other hand, if \eqref{eq:morethanhalf} does not hold (bad case), then, from the first part of Lemma \ref{Lm:Barrier}, we have, by explicit computations, the decrease of our control, i.e.,
\[
s_{k+1}\geq s_k(1-C_0\lambda).
\]
Since \eqref{eq:morethanhalf} holds half of the time, the final control we get on $s_k$ will depend on the order of the $t_k$'s where \eqref{eq:morethanhalf} is satisfied. Clearly, the worst case scenario is when it holds for the first half of times $t_k$. We only present the estimate in this case.

We start with $s_0=0$. Then, we apply the second part of Lemma \ref{Lm:Barrier} $\bar k$ times, where $\bar k=\lceil1/(4\lambda)\rceil$ represents half of the amount of values $t_k$, to get
\[
s_{\bar k}=\bar kc_0\lambda \ge\frac{c_0}{4}.
\]
In the remaining $\bar k$ times we are in the bad case, where we use the first part of Lemma \ref{Lm:Barrier}. We then have the final estimate on the control
\begin{align*}
    s_{2\bar k}\geq s_{\bar k}(1-C_0\lambda)^{\bar k}\geq \bar c,
\end{align*}
provided that $\lambda<
1/C_0$. Thus, this lemma implies
\begin{align*}
    \bar w(x,t)>\bar c\, \phi(x),\quad  (x,t) \in Q^+_1\times[-1/2,0].
\end{align*}
which gives the desired estimate since $\phi>c$ in $Q_{1/2}^+$.
\hfill $\square$

By virtue of Proposition \ref{prop:wharnack}, we can prove that $w$ has an \textit{almost} H\"older-modulus of continuity, which will play a key role in the improvement of flatness.
\begin{Proposition}[Almost H\"older continuity for $w$] \label{Pro:w_holder}
    Assume that $u$ satisfies the hypotheses of Proposition \ref{Prop:ImprovFlat} and the error $w$ is always defined by \eqref{eq:w_error}. Then, for any $(x_0,t_0) \in \mathcal{C}_{1/2}$, it holds
    \begin{equation*}
        \osc\limits_{\mathcal{B}_{\lambda,r}(x_0,t_0)} w \leq C r^{\alpha}
    \end{equation*}
    for $r > C(\delta) \varepsilon^{1/2}$, with $\delta \leq K^{-1}$ and $\varepsilon \leq c_1 \delta$.
\end{Proposition}
 The proof follows from Proposition \ref{prop:wharnack} as in the homogeneous case. For the sake of completeness, we provide the details of the proof, which relies on the metric $d_\lambda$. \\
\begin{proof}
Note that applying Proposition \ref{prop:wharnack} to $u$, by \eqref{eq:w_error}, we have 
\begin{equation*}
    	(c_1^{(1)}-1)\varepsilon \lambda \leq u-l_{a,b} \leq \varepsilon \lambda(1-c_2^{(1)})\quad\text{in }\mathcal{C}_{\lambda/2},
    	\end{equation*}
with $c_1^{(1)},c_2^{(1)} \geq 0$, such that $2 c= c_1^{(1)}+c_2^{(1)}$ where $c$ is the one given by Proposition \ref{prop:wharnack}.
If we consider 
\begin{equation*}
    \tilde w_1 \left( \frac{2x}{\lambda}, \frac{2t}{\lambda}\right):= \frac{(u - \ell_{a,b})(x,t)}{\varepsilon \lambda} -K_1,\quad (x,t)\in \mathcal{C}_{\lambda/2},
\end{equation*}
where $K_1:=\frac{1}{2} (c_1^{(1)} - c_2^{(1)})$, it follows that $-1+c \leq \tilde w_1 \leq 1- c$. Now, defining  $ w_1:= \frac{\tilde w_1}{1-c}$, we have $w_1 \in (-1,1)$ and according to \eqref{eq:w_error} we can write $w_1$ as

\begin{equation*} 
    w_1\left( \frac{2x}{\lambda}, \frac{2t}{\lambda}\right) = \frac{(u - \ell_{a,b})(x,t) - \varepsilon\lambda K_1}{\lambda \varepsilon (1- c)} = \frac{(u - \ell_{a,b_1})(x,t)}{\lambda_1 \varepsilon_1}, \quad (x,t) \in \mathcal{C}_{\lambda/2},
\end{equation*}
where we have called
\[
\quad \varepsilon_1 := 2(1- c)\varepsilon, \quad \lambda_1:= \frac{\lambda}{2}, \quad b_1:=b+\varepsilon \lambda K_1. 
\]
We are in place to invoke again Proposition \ref{prop:wharnack}, which implies
\[
 \osc_{\mathcal{C}_{1/2}} w_1 \leq 2(1-c)
\]
if $\varepsilon_1 \leq c_1 \delta$ and $2\varepsilon^{1/2} \leq c''$, for some $c''$ small universal for which it is possible to apply Lemma \ref{Lm:Barrier}, since the $H(\varepsilon^{1/2})$-property for $u$ in $\mathcal{C}_{\lambda}$ implies the $H(2\varepsilon^{1/2})$-property for $u$ in $\mathcal{C}_{\lambda/2}$.
In order to iterate this procedure, we introduce
\begin{equation*}
    w_k\left( \frac{2^k x}{\lambda}, \frac{2^k x}{\lambda}\right) := \frac{(u - \ell_{a,b_k})(x,t)}{\varepsilon_k \lambda_k} \quad (x,t) \in C_{\lambda_k}
\end{equation*}
defining 
\[
\varepsilon_k:=2^k(1-\ c)^k \varepsilon, \quad \lambda_k:= \frac{\lambda}{2^k}, \quad b_k:= b_{k-1} + \varepsilon_{k-1}\lambda_{k-1} K_k,
\]
where $K_k:=\frac{1}{2}( c_1^{(k)}-c_2^{(k)})$, with $c_1^{(k)}, c_2^{(k)} \geq 0, \, c_1^{(k)} + c_2^{(k)}=2c$, given by the relation
\[
(-1 +c_1^{(k)}) \varepsilon_{k-1} \lambda_{k-1} \leq u - \ell_{a,b_{k-1}} \leq (1 -c_2^{(k)}) \varepsilon_{k-1} \lambda_{k-1}, \quad (x,t) \in C_{\lambda_k}.
\]
Thus, we can iterate Proposition \ref{prop:wharnack} applied to $w_k$ as long as $2^k (1- c) \leq c_1 \delta$ and $2^k \varepsilon^{1/2} \leq c''.$ In terms of $w$, this reads as its oscillation in $\mathcal{C}_{2^{-k}}$ is bounded by $2(1 - c)^{k}$. Indeed
\begin{align*}
2(1-c) &\geq \osc_{\mathcal{C}_{1/2}} w_k = \osc_{\mathcal{C}_{\lambda/2^k}} \left(\frac{u-\ell_{a,b_k}}{\varepsilon_k \lambda_k} \right) =  \osc_{\mathcal{C}_{\lambda/2^k}} \left(\frac{u-\ell_{a,b}}{\varepsilon_k \lambda_k} \right) \\
&=\osc_{\mathcal{C}_{1/2^k}} w\left(\frac{\varepsilon \lambda}{\varepsilon_k \lambda_k} \right) = \osc_{\mathcal{C}_{1/2^k}} w \frac{1}{(1- c)^k} 
\end{align*}

We now distinguish two cases. If $\nor{(x_0)_n} \leq r$, then the construction of $\mathcal{B}_{\lambda,r}(x_0,t_0)$, see \eqref{eq:balls-d_lambda}, coincides with $\mathcal{C}_{r}(x_0,t_0)$. So, the iteration leads to the desired result if $r\geq C(\delta)\varepsilon^{1/2}$, where $C(\delta)$ satisfies $2^{k_0}\le C(\delta)\le 2^{k_0+1}$. Otherwise, for interior balls  $\mathcal{B}_{\lambda,r}(y,s)$, the same diminishing behavior of the oscillation holds by using the $H(\sigma)$-property for $w$ directly, as discussed in Subsection \ref{The inhomogeneousHproperty}.
 \end{proof}
\begin{Lemma}[$w$ is close to a H\"older function] \label{w_near_holder}
For every $\bar \varepsilon > 0$, there exists $\bar \delta >0$ such that if $w$ satisfies Proposition \ref{Pro:w_holder} for $r \geq \bar \delta$, $-1\leq w \leq 1$, its graph is compact in $\mathbb{R}^{n+2},$ and \(w\) is single-valued for $x_n \geq \bar \delta,$ then there exists a (single-valued) function $\phi \in C_{d_{\lambda}}^{\alpha}( {\overline{\mathcal{C}_{1/2}}})$ such that 
\begin{equation*}
    \nor{w- \phi} \leq \bar \varepsilon, \qquad \norm{\phi}_{C_{d_{\lambda}}^{\alpha}( {\overline{\mathcal{C}_{1/2}}})} \leq C.
\end{equation*}
\end{Lemma}
In order to obtain this result, we report the subsequent lemma, which is a slight variation of \cite[Lemma 9.1]{KS} for multi-valued functions with compact graphs.  In the following, $d_{\mathcal{H}}(\cdot,\cdot)$ will denote the Hausdorff distance of two nonempty sets, defined as 
\begin{equation*}
    d_{\h}(K_1,K_2):= \max \left\{  \sup_{ y\in K_1} d( y, K_2), \sup_{ y\in K_2} d( y, K_1) \right\},
\end{equation*}
where $K_1,K_2 \in \R^{n}$.
Moreover, if $(X,d)$ is a metric space, $\mathcal{K}(X)$ denotes the set of all nonempty compact subsets of $X$. 
Furthermore, if $X$ is compact, then so is $\mathcal{K}(X)$.

\begin{Lemma} \label{Lm:w_approx}
    Let $\{w_k\}_{k \geq 1}$ be a sequence of possibly multi-valued continuous functions on  $\overline{\mathcal{C}_{1/2}} \subset \mathbb{R}^{n+1}$ so that:
    \begin{itemize}
        \item[(i)] $\nor{w_k} \leq 1$ for all $k \geq 1$; 
        \item[(ii)] there exists a sequence $\delta_k \rightarrow 0$ as $k \rightarrow \infty$ such that, for each $k \geq 1$,  $w_k$ is single-valued in the region where $x_n \geq \delta_k$ and there exists a modulus of continuity $\omega : [0,\infty) \rightarrow [0,\infty) $ for which
        \begin{equation*} \osc\limits_{{\mathcal{B}_{\lambda,d_p((x,t),(y,s))}(x,t)}} w_k \leq  \omega(d_p((x,t),(y,s))),
        \end{equation*}
        for all  $(x,t),(y,s) \in \overline{\mathcal{C}_{1/2}}$  such that $ d_p((x,t),(y,s)) \geq \delta_k$. 
    \end{itemize}
    For any $k \geq 1$, let $G(w_k):=\left\{ (x,t,x_{n+1}) :\,\,(x,t) \in \overline{\mathcal{C}_{1/2}}, \,\, x_{n+1} \in w_{k}(x,t) \right\}$ be the graph of $w_k$ in $\mathbb{R}^{n+2}$. Then, there exists a single-valued function $\phi : \overline{\mathcal{C}_{1/2}}\rightarrow [-1,1]$ such that 
    \begin{equation*} \label{eq:Lm:w_approx}
        d_{\mathcal{H}}\left(G(w_k), G(\phi) \right) \rightarrow 0 \quad \text{for }k \rightarrow \infty.
    \end{equation*}
    Moreover, for all  $(x,t),(y,s) \in \overline{\mathcal{C}_{1/2}}$, we have
    \begin{equation*}
        \nor{\phi(x,t)-\phi(y,s)} \leq \omega\left( d_p((x,t),(y,s))\right).
    \end{equation*}
\end{Lemma}
\noindent\textit{Proof of Lemma~\ref{Lm:w_approx}.~}
  By the continuity of $w_k$ and $(i)$, we have that $G_k\equiv G(w_k)$ is a compact subset of $\overline{\mathcal{C}_{1/2}} \times [-1,1]$. Since $\mathcal{K}(\overline{\mathcal{C}_{1/2}} \times [-1,1])$ is a compact metric space with the Hausdorff distance, there exists a compact set $K\subset \overline{\mathcal{C}_{1/2}} \times [-1,1]$ such that
\begin{equation}\label{eq:convhauss}
d_{\h}(G_k, K)\to 0 \quad\hbox{as}~ k \to \infty,
\end{equation}
up to a subsequence.

Next, we show that $K$ must be the graph of some (single-valued) function. If not, there exist two points $(x,t,p), (x,t,q) \in K$ with $p\neq q$. By \eqref{eq:convhauss}, for all $\varepsilon>0$, there is $(k_0)_{\varepsilon}\geq 1$ such that $d((x,t,p), G_k) \leq \varepsilon$ for all  $k\geq (k_0)_{\varepsilon}$ (here, $d$ denotes the distance in $\mathbb{R}^{n+2}$).
In particular, being $G_k$ compact, the distance is attained at some point $(x_k,t_k, p_k)\in G_k$. Hence,
\begin{equation*}
(x_k,t_k, p_k) \to (x,t,p) \quad\ \hbox{as}~ k \to \infty.
\end{equation*}
Similarly, there are points $(y_k,s_k, q_k)\in G_k$ such that
$$
(y_k, s_k , q_k) \to (x,t, q) \quad\ \hbox{as}~ k \to \infty.
$$
Consider first the case in which $x_n > 0$. Choosing $k$ large enough, we can always assume that $(x_k)_n \geq \delta_k$. We distinguish two cases. If  $d_p((x_k,t_k),(y_k,s_k))\geq \delta_k$  for infinitely many $k$'s, then by $(ii)$, we get
\begin{align*}
|p_k-q_k| & \leq  \osc_{\substack{\mathcal{B}_{\lambda,d_p((x_k,t_k),(y_k,s_k))}(x_k,t_k)}} w_k \leq  \omega (d_p((x_k,t_k),(y_k,s_k))).
\end{align*}
Letting $k\to\infty$, it holds that $p=q$, which is a contradiction.
Otherwise, for sufficiently large $k$, there exists some $(z_k,r_k) \in \overline{\mathcal{C}_{1/2}} $ such that $\delta_k \leq {d_p((x_k,t_k),(z_k,r_k))}\leq 3\delta_k$ and $\delta_k \leq {d_p((z_k,r_k),(y_k,s_k))}\leq 3\delta_k$.
Then,
\begin{align*}
|p_k-q_k| \leq \osc_{\substack{\mathcal{B}_{\lambda,{d_p((x_k,t_k),(z_k,r_k))}}(x_k,t_k)}} w_k  + \osc_{\substack{\mathcal{B}_{\lambda,{d_p((z_k,r_k),(y_k,s_k))}}(y_k,s_k)}} w_k \leq 2 \omega (3\delta_k){.}
\end{align*}
Letting $k\to\infty$, we again reach a contradiction.

Observe that for any $(x,t) \in \overline{\mathcal{C}_{1/2}}$, there is some $p\in [-1,1]$ such that $(x,t,p)\in K$. Indeed, let $(x,t)\in \overline{\mathcal{C}_{1/2}}$ and consider $(x,t,w_k(x,t))\subset G_k$. By \eqref{eq:convhauss}, for all $\varepsilon>0$, there exists ${(k_0)_{\varepsilon}}\geq 1$ such that ${d((x,t,w_k(x,t)), K)} \leq \varepsilon$ for all $k\geq {(k_0)_{\varepsilon}}$. Taking $\varepsilon= 1/j$, by a diagonalization argument, we obtain that there is a subsequence of $\{(x,t,w_k(x,t))\}_{k\geq 1}$ converging to $(x,t,p)\in K$, where $|p|\leq 1$ by $(i)$.\\
We define the function $\phi: \overline{\mathcal{C}_{1/2}}\to [-1,1]$ as $\phi(x,t):=p$, if $x_n > 0$.
Let $(x,t){\neq} (y,s) \in \overline{\mathcal{C}_{1/2}}$. Then,
$$
|\phi(x,t)-\phi(y,s)| \leq \lim_{k\to\infty} \osc_{\substack{\mathcal{B}_{\lambda,\delta}(x,t)}} w_k \leq \omega({\delta}),
$$
where we used that $d_p((x,t),(y,s))=:\delta \geq \delta_k>0$ for $k$ large enough. This proves that $\phi $ is uniformly continuous with a modulus $\omega.$
 To also include the case $x_n=0$, we point out that $w_k(x',\varepsilon,t) \rightarrow \phi(x',\varepsilon,t)$ as $k \rightarrow \infty$ for every $\varepsilon >0$. Being $\phi$ bounded in $\mathcal{C}_{1/2}$ {by definition}, we define $\phi(x',0,t):= \lim_{\varepsilon \rightarrow 0^+} \phi(x',\varepsilon,t)$, where the limit exists up to a subsequence.
\hfill $\square$

\noindent\textit{Proof of Lemma~\ref{w_near_holder}.~}
	We argue by contradiction by assuming that there exist sequences of $\delta_k \rightarrow 0$ and $w_k$ satisfying the assumption with $\delta_k$ instead of $\bar \delta$. Therefore, there exists $\varepsilon_0 > 0$ such that 
\begin{equation*}
    \norm{w_k - \phi}_{L^{\infty}(\mathcal{C}_{1/2})} >\varepsilon_0
\end{equation*}
holds for every $\phi \in C_{d_{\lambda}}^{\alpha}( \mathcal{C}_{1/2})$.
However, this immediately leads to a contradiction by applying Lemma \ref{Lm:w_approx}.
\hfill $\square$

\section{Proofs of the main results}
\label{Section:C1a}
\noindent\textit{Proof of Proposition~\ref{Prop:ImprovFlat}.~}
Since the limiting problem, \eqref{eq:LimitStefanNH}, coincides with the one in the homogeneous case, this proof is identical to the one presented in \cite{DFS23}, having obtained the analogous results for the error function $w$ in the inhomogeneous case, see Subsection \ref{Section:Osc_w}. Anyway, for the sake of completeness, we include the proof. We proceed in two steps.\\ 
    \textbf{Step 1. } 
    First, we prove that there exists a function $v$ solving \eqref{eq:LimitStefanNH}, which approximates well $w$ in $\mathcal{C}_{1/2}$, that is,
    \begin{equation}\label{eq:v_approx_w}
        \nor{v-w} \leq C \delta \quad \text{in } \mathcal{C}_{1/2},
    \end{equation}
    as long as $\varepsilon \leq \varepsilon_1(\delta)$.\\
    To this end, we want to apply Lemma \ref{w_near_holder} to $w$ with $\bar \varepsilon=\delta$. We thus take $\varepsilon$ so small that $\varepsilon^{1/2}\leq\bar \delta$, only depending on $\delta$ and universal constants, with $\bar \delta$ given in Lemma \ref{w_near_holder}. We remark that our assumptions imply that we are in the conditions of the lemma, which guarantees the existence of a (single-valued) function $\phi \in C_{d_{\lambda}}^{\alpha}( {\overline{\mathcal{C}_{1/2}}})$ such that
\begin{equation}\label{eq:phi_approx_w}
    \nor{w- \phi} \leq \delta, \qquad \norm{\phi}_{C_{d_{\lambda}}^{\alpha}( {\overline{\mathcal{C}_{1/2}}})} \leq C.
\end{equation}
In view of Proposition \ref{Prop:Dirichletv}, let $v$ be the solution of \eqref{eq:LimitStefanNH} in $\mathcal{C}_{1/2}$ with $v = \phi $ on $\p_D \mathcal{C}_{1/2}$, satisfying
\begin{equation}\label{eq:estim_Holder-norm_v}
	\norm{v}_{C_{d_{\lambda}}^{\alpha}( \overline{\mathcal{C}_{1/2}})} \leq C.
\end{equation}
From this, together with \eqref{eq:phi_approx_w}, if $d_\lambda((x,t), \p_D \mathcal{C}_{1/2}) \leq \delta^{1/\alpha}$, then there exists $(y,s) \in \p_D \mathcal{C}_{1/2}$ for which
{\begin{align*}
	&\nor{v(x,t)- \phi(y,s)}\leq C_1 \delta,\\
    &\nor{w(x,t)- \phi(y,s)} \leq \nor{w(x,t)- \phi (x,t)}+\nor{\phi (x,t)- \phi (y,s)}\leq \delta + C d_\lambda((x,t), (y,s))^{\alpha} \leq C_2 \delta.
\end{align*}}
Therefore, we have
\begin{equation}\label{eq:v_approx_w_aux}
|v-w|\leq C\delta\quad \text{in }\mathcal{C}_{1/2} \cap \{(x,t):\,\,d_\lambda((x,t), \p_D \mathcal{C}_{1/2}) \leq \delta^{1/\alpha}\} .
\end{equation}
In particular, for $\Omega := \mathcal{C}_{1/2} \cap \{d_\lambda((x,t)), \p_D \mathcal{C}_{1/2}) > \delta^{1/\alpha} \} $, it holds
\[
|v-w|\leq C \delta \quad \textit{ on } \,{\partial_D \Omega}.
\]
Using the estimates of Proposition \ref{Prop:Dirichletv}, we obtain $\nor{\nabla v }, \nor{D^2 v} \leq C(\delta)$ in $\Omega$, recalling \eqref{eq:estim_Holder-norm_v}. Consequently, by Corollary \ref{Cor:comparison}, it holds
\begin{equation*}
     \nor{v- w} \leq C \delta \quad \textit{in } \Omega,
\end{equation*}
which proves \eqref{eq:v_approx_w}, together with \eqref{eq:v_approx_w_aux}. \\
\textbf{Step 2.} Now, applying Proposition \ref{Prop:Interiorv} to the solution $v$ above, we achieve
\begin{equation*}
    \nor{w-\ell_{\bar a, \bar b}} \leq \nor{w-v}+ \nor{v-\ell_{\bar a, \bar b} } \leq C \delta+\bar C\rho^{1+\alpha}\quad \text{in } \mathcal{C}_{\rho},\quad {0<\rho\le \frac{1}{2}}.
 \end{equation*}
{We then choose $\tau$ small and universal in such a way that if $\rho=\tau$ and $\delta= \tau^{1+\alpha/2}$, we have $\delta \leq c'$, with $c'$ given by Proposition \ref{Pro:w_holder}, and} 
\begin{equation*}
    \nor{w- \ell_{\bar a, \bar b}} \leq \frac{1}{4}\tau \quad \text{in }\mathcal{C}_{\tau}, \quad \nor{\bar a_n'} \leq \frac{1}{4}\tau^{-2} \delta \lambda^{-1}.
\end{equation*}
{Going back to the original function $u$, see \eqref{eq:w_error}, we get
\begin{equation*}
	\nor{(u - \ell_{a, b})(\lambda x,\lambda t) - \varepsilon \lambda \ell_{\bar a, \bar b}(x,t)}\leq \frac{\varepsilon}{4}\tau \lambda \quad \text{in } \mathcal{C}_{\tau}, 
\end{equation*}
which gives
\begin{equation*}
	\nor{(u - \ell_{a, b})(x,t) - \varepsilon \lambda \ell_{\bar a, \bar b}\bigg(\frac{x}{\lambda}, \frac{t}{\lambda}\bigg)}\leq \frac{\varepsilon}{4}\tau \lambda \quad \text{in } \mathcal{C}_{\tau\lambda}. 
\end{equation*}}
Setting 
\begin{equation*}
    \tilde a(t):= a(t) + \varepsilon \bar a \left( \frac{t}{\lambda} \right), \quad \hat b(t) := b(t) + \varepsilon \lambda \bar b\left( \frac{t}{\lambda} \right),
\end{equation*}
we can rewrite the previous inequality as
\begin{equation*}
    \nor{u - \ell_{\tilde a, \hat b}} \leq \frac{\varepsilon}{4}\tau \lambda \quad \text{in } \mathcal{C}_{\tau \lambda}, \quad \nor{\tilde a_n'} \leq \frac{\varepsilon \delta }{\lambda^2 } \left( 1 + \frac{1}{4\tau^2}\right) \leq \frac{\varepsilon \delta}{2 (\tau \lambda)^2}.
\end{equation*}
To have a control on the behavior of $b$, we define $\tilde b$ as the solution of the ODE
\begin{equation*}
    \tilde b' = g(\tilde a), \quad \tilde b(0) = \hat b(0),
\end{equation*}
for which we have
\[
\hat b'(t)=b'(t)+\varepsilon\bar b'\left(\frac{t}{\lambda}\right)=g(a(t))+\varepsilon\nabla g(a(t))\cdot \bar a\left(\frac{t}{\lambda}\right)=g(\tilde a(t))+O(\varepsilon^2)=\tilde b'(t)+O(\varepsilon^2).
\]
Therefore, if $t\in [-\tau \lambda,0],$ then
\[
|(\tilde b- \hat b)(t)|\leq C\varepsilon^2|t|\leq C\varepsilon^2\tau \lambda.
\]
This allows us to estimate
{\begin{align*}
    \nor{u-\ell_{\tilde a, \tilde b}}&=\nor{u-\tilde a(t)\cdot x-\tilde b(t)}\\
    &\leq\,\nor{u-\tilde a(t)\cdot x- \hat b(t))}+\nor{\tilde b(t) - \hat b(t)}\\
    &\leq\, \frac{\varepsilon}{4}\tau \lambda+C\varepsilon^2\tau \lambda\\
    &\leq \, \frac{\varepsilon}{2}\tau \lambda
\end{align*}}
for $\varepsilon\leq \varepsilon_1$ small enough, as intended, concluding the proof.
\hfill $\square$
\vspace{5pt}

\noindent\textit{Proof of Theorem~\ref{Theorem:flatsolution}.~}
    As already pointed out in Remark \ref{Rmk:flatimpliesreg}, Theorem \ref{Theorem:flatsolution} is equivalent to obtaining $C^{1,\alpha} $-estimates up to the boundary $\{x_n = 0 \}$ for the hodograph transform. In the same remark, we highlighted how the hypotheses of Theorem \ref{Theorem:flatsolution} translate into Proposition \ref{Prop:ImprovFlat}. In particular, we can take $\lambda \leq \min\{\delta \varepsilon_1, \lambda_1\}$ and assume that $a_0(t)=(0,\hdots, 0, (a_0)_n(t))\in \mathcal{R}_{K/2} $.\\
    We can thus iterate Proposition indefinitely \ref{Prop:ImprovFlat} in cylinders {$\mathcal{C}_{\lambda_k}$ with $\lambda_k:=\lambda \tau^k$} and  $\varepsilon=\varepsilon_k:=\varepsilon_1 2^{-k}.$ If we write $k=(\log_\tau \lambda_k - \log_\tau \lambda),$ then
 \begin{equation*}
     {\varepsilon_k=\varepsilon_1 \left( \frac{\lambda_k}{\lambda} \right)^{-\log_{\frac{\lambda_k}{\lambda}}2\,\cdot\,\log_{\tau}\frac{\lambda_k}{\lambda}}}=\varepsilon_1 \left( \frac{\lambda_k}{\lambda} \right)^{-\frac{1}{\log_2 \tau}} = \varepsilon_1 \lambda^\frac{1}{\log_2 \tau} \lambda_{k}^{-\frac{1}{\log_2 \tau}}=: C(\lambda) \lambda_k^{\alpha}.
 \end{equation*}
 Furthermore, the given coefficients $a_k(t)$ are still in $\mathcal{R}_{K/2}$, for every $k$,  since it holds
 \[
 \nor{a_k(t)-a_0(t)}\leq \sum_{j=1}^k\nor{a_j(t)-a_{j-1}(t)}\leq C\sum_{j=1}^k\varepsilon_{{j-1}}\leq C\varepsilon_1.
 \]
In addition,
\begin{equation}\label{eq:approx_u_Thm_1.1}
|u-\ell_{a_k,b_k}|\leq \varepsilon_k\lambda_k= C(\lambda)\lambda_k^{1+\alpha} \quad \text{in } \mathcal{C}_{\lambda_k}
\end{equation}
for all $k\geq 0$, which implies pointwise $C^{1,\alpha}$-regularity in space at $0$.

We also have{, recalling the definition of $\lambda_k$,}
\begin{align*}
    \nor{a_{k+n}{(t)}- a_k{(t)}} &\leq \sum_{j=1}^{n} \nor{a_{k+j}{(t)} - a_{k+j-1}{(t)}} \leq C \lambda^{\alpha} \sum_{j=0}^{n-1} \tau^{({k+j}) \alpha}  \\
    &\leq  C \lambda^{\alpha}\tau^{k \alpha}  \sum_{j=0}^{\infty} \tau^{ j \alpha} = C \lambda_k^{\alpha} \frac{1}{1-\tau^{\alpha}} \leq C \lambda_k^{\alpha},
\end{align*}
for $\tau \leq 2^{-\alpha}$.
So, $a_k{(t)}$ is a Cauchy sequence, and for $k\rightarrow \infty$, it holds $a_k(t)\rightarrow a_\infty{(t)}$ with $a_\infty{(t)}= \nabla u (0,t)$. Then, 
\begin{equation*}
	\nor{\nabla u (0,t) - a_k(t)} \leq \sum_{j={k}}^{\infty} \nor{a_j(t) - a_{j+1}(t)} \leq C \lambda_k^{\alpha}.
\end{equation*}

Now, if $t,s$ are such that $\nor{t-s}=\lambda_k^2$, we get, from \eqref{eq:approx_hodog_linear_const_time},
\begin{align*}
    \nor{\nabla u (0,t) - \nabla u (0,s)} &\leq   \nor{\nabla u (0,t) - a_k (s)} + \nor{\nabla u (0,s) - a_k (s)} \\
   &\leq C \lambda_k^{\alpha}+2\varepsilon_{1}\lambda_k = C_1\nor{t-s}^{\alpha/2}.
\end{align*}
Thus, for any $t,s \in [-\lambda_k, 0]$,
\begin{equation}\label{eq:01}
    \begin{aligned} 
    \nor{a_k(t)-a_k(s)} \leq \,&\nor{a_k(t) - \nabla u (0,t)} + \nor{\nabla u (0,t) -\nabla u (0,s)} + \nor{\nabla u (0,s)- a_k(s)}\\
    \leq\,&  C\lambda_k^\alpha + {C_1}\lambda_k^{\alpha/2}+C\lambda_k^\alpha  \leq {C_2}\lambda_k^{\alpha/2}.
\end{aligned}
\end{equation}
Being $b_k'{(t)}=g(a_k{(t)}) \in C^{1}$, by  
\eqref{eq:boundg} and \eqref{Hp:condit-assum-flat}, we have
{\begin{equation*}
    \nor{b''(t)} \leq \bar C \lambda_k^{\alpha-2},
\end{equation*}}
which yields, with $t \in (t_0 - \lambda_k^2, t_0 + \lambda_k^2)$, we get
\begin{equation*}\
    \nor{b_k(t)- b_k(0) - b_k'(0) t} \leq  C_3 \lambda_k^{\alpha+2}.
\end{equation*}
Finally, exploiting this and \eqref{eq:01}, we reach, according to \eqref{eq:approx_u_Thm_1.1},
\begin{align*}
    \nor{u-(a_k(0) \cdot x +b_k'(0)t + b_k(0))} &\leq \nor{u - \ell_{a_k,b_k} }  + \nor{  \ell_{a_k,b_k}  -  \ell_{a_k,b_k}{(x,0)} - b_k'(0)t} \\
    & \leq  C \lambda_k^{1+\alpha} + \nor{(a_k(t)-a_k(0))\cdot x } + \nor{b_k(t)- b_k(0) - b_k'(0) t}  \\
    & \leq  C \lambda_k^{1+\alpha} +  {C_2} \lambda_k^{1+\alpha/2} + {C_3} \lambda_k^{1+\alpha/2}\leq  C(\lambda) \lambda_{k}^{1+\alpha/2}\quad \text{in } C_{\lambda_k}.
\end{align*}
\hfill $\square$

\section*{Acknowledgements}
This research is partially supported by PRIN 2022 7HX33Z - CUP J53D23003610006 and by University of Bologna funds
\say{Attività di collaborazione con università del Nord America} in the framework of the project:
 \say{Interplaying problems in analysis and geometry}. F.F. and D.G. are also partially supported by  INDAM-GNAMPA project 2025: \say{At The Edge Of Ellipticity} - CUP E5324001950001. 
 N.F. is partially supported by  INDAM-GNAMPA project 2023 \say{Problemi variazionali/nonvariazionali: interazioni tra metodi integrali e principi del massimo}.
The authors would like to thank Ovidiu Savin and Daniela De Silva for fruitful
conversations on the topic of this paper. 
D.G. and D.J. wish to thank the Department of Mathematics of Columbia University for the warm hospitality.
Finally, the authors wish to thank the anonymous referee for the valuable suggestions, which improved the presentation of the manuscript.

\section*{Competing interests to declare}
There are no relevant financial or non-financial competing interests to report.

\bibliographystyle{abbrv}
\bibliography{BibStefan.bib}
\Addresses
\end{document}